\documentclass{article}

\usepackage{graphicx}
\usepackage{fullpage}
\usepackage{color}
\usepackage{array,cite}
\usepackage{fancyhdr}
\usepackage{amsmath,amsthm}
\usepackage{amssymb,latexsym}
\usepackage{mathrsfs}
\usepackage{cancel}
\usepackage{xifthen}
\usepackage{MnSymbol}
\usepackage{appendix}

\usepackage[bookmarks=false,pdfstartview={FitH}]{hyperref}
\usepackage{url}
\usepackage{enumitem}
\usepackage{bbm}
\usepackage[FIGTOPCAP]{subfigure}
\usepackage{verbatim}

\newtheorem{theorem}{Theorem}[section]
\newtheorem{lemma}[theorem]{Lemma}
\newtheorem{proposition}[theorem]{Proposition}

\theoremstyle{definition}

\theoremstyle{remark}
\newtheorem{remark}[theorem]{Remark}

\def\XXint#1#2#3{{
\setbox0=\hbox{$#1{#2#3}{\int}$}
\vcenter{\hbox{$#2#3$}}\kern-.5\wd0}}

\newcommand{\cc}[1]{C_{#1}}

\newcommand{\allrip}[1]{\mathcal{C}_{#1}}
\newcommand{\tallrip}[1]{\tilde{\mathcal{C}}_{#1}}
\newcommand{\mneta}{\mathcal{M}_{N,\eta}}
\newcommand{\eucball}{\mathscr{B}(y,\epsilon)}
\newcommand{\vne}{\mathcal{G}_{N,\epsilon}(y)}
\newcommand{\sve}[1]{S_{N-j}^\circ({#1})}

\newcommand{\bc}{\mathfrak{B}_{\mathcal{C}}}
\newcommand{\bstarc}{\mathfrak{B}^*_{\mathcal{C}}}
\newcommand{\tbc}{\tilde{\mathfrak{B}}_{\mathcal{C}}}

\newcommand{\indentitem}{\setlength\itemindent{12pt}}

\newcommand{\tds}{\tilde{D}_s}

\newcommand*\samethanks[1][\value{footnote}]{\footnotemark[#1]}

\begin{document}

\title{Entropic repulsion of Gaussian free field on high-dimensional Sierpinski carpet graphs}

\author{
Joe P. Chen\thanks{
Research partially supported by a visiting scholarship from the Hausdorff Research Institute for Mathematics.} \thanks{
Research partially supported by US NSF grant DMS-1156350.} \\ \small{University of Connecticut}
\and
Baris Evren Ugurcan\samethanks \\ \small{University of Western Ontario} 
}

\maketitle

\renewcommand{\theequation}{\thesection.\arabic{equation}}
\numberwithin{equation}{section}

\begin{abstract}
Consider the free field on a fractal graph based on a high-dimensional Sierpinski carpet
(\emph{e.g.} the Menger sponge), that is, a centered Gaussian field whose covariance is the Green's function for simple random walk on the graph.
Moreover assume that a ``hard wall" is imposed at height zero so that the field stays positive everywhere.
We prove the leading-order asymptotics for the local sample mean of the free field above the hard wall on any transient
Sierpinski carpet graph, thereby extending a result of Bolthausen, Deuschel, and Zeitouni for the free field on
$\mathbb{Z}^d$, $d \geq 3$, to the fractal setting. 


\end{abstract}

\maketitle

\section{Introduction} \label{sec:intro}

Let $\mathcal{G}=(V,E)$ be a connected graph of bounded degree, which contains a distinguished vertex $x_0 \in V$. Let $(Y_k)_{k \in \mathbb{N}_0}$ be a discrete-time symmetric random walk on $\mathcal{G}$, with transition probability
$$P(x,y) = \mathbb{P}^x[Y_{k+1}=y | Y_k=x]=\left\{\begin{array}{ll}1/\deg(x),& \text{if}~\langle xy\rangle \in E,\\ 0,& \text{if}~\langle xy \rangle \notin E.\end{array}\right.$$
As usual, $\mathbb{P}^x$ denotes the law of the random walk started from $x$, and $\mathbb{E}^x$ denotes the corresponding expectation. Then we define the Green's function for the random walk killed at $x_0$ by
\begin{equation}
G_{\mathcal{G}}(x,y) = \frac{1}{\deg(y)} \mathbb{E}^x \left[ \sum_{k=0}^{\tau_{x_0}-1} \mathbbm{1}_{\{Y_k=y\}}\right] \quad (x,y \in V),
\label{eq:Greenfunctiondef}
\end{equation}
where $\tau_{x_0} = \inf\{k \geq 0: Y_k=x_0\}$ is the first hitting time of  $x_0$. 

A \emph{Gaussian free field (GFF)} on $\mathcal{G}$ is a collection of Gaussian random variables $\varphi=\{\varphi_x\}_{x\in V}$ with mean zero and covariance given by the Green's function $G_{\mathcal{G}}(x,y)$ in (\ref{eq:Greenfunctiondef}). Formally, the law of $\varphi$ has density proportional to $\exp\left(-\frac{1}{2}E_{\mathcal{G}}(\cdot)\right)$ with respect to the Lebesgue measure on $\mathbb{R}^V$, where
$$E_{\mathcal{G}}(f) = \frac{1}{2}\sum_{\substack{x,y\in V\\ \langle x y \rangle \in E}}[f(x)-f(y)]^2\quad (f : V\to\mathbb{R}~\text{with}~f(x_0)=0)$$
is the Dirichlet energy on $\mathcal{G}$.

From the perspective of statistical mechanics, it is helpful to view the free field $\varphi$ as a \emph{random height function} on $V$, or a \emph{random surface} embedded in $V\times \mathbb{R}$. We would like to study the stochastic geometry of the free field on an infinite graph, under the condition that $\varphi$ is nonnegative everywhere: that is, nowhere can $\varphi$ penetrate downward through a ``hard wall'' imposed at zero height. The phenomenon of \emph{entropic repulsion} then refers to the event that the random field moves away from the hard wall in order to gain space for local fluctuations, and especially to accommodate the largest downward spikes.

Entropic repulsion has been demonstrated for a large class of Gaussian random fields on $\mathbb{Z}^d$, which we
now recap briefly. Let $\varphi$ be the centered Gaussian field on $\mathbb{Z}^d$ with covariance $(-\Delta)^{-p}$,
where $\Delta$ is the probabilistic Laplacian on $\ell^2(\mathbb{Z}^d)$ and $p\in \mathbb{N}$, and denote the law
of this random field by $\mathbb{P}$. The $p=1$ case corresponds to the free field, while the $p=2$ case is
called the Laplacian membrane model \cite{Sakagawa, Kurt}, as the random surface favors minimization of
curvature rather than gradient. Furthermore, for each $L\in 2\mathbb{N}$, let $\Lambda_L =
[-\frac{L}{2},\frac{L}{2}]^d \cap \mathbb{Z}^d$, and $\Omega^+_{\Lambda_L}$ be the event $\{\varphi_x \geq 0
~\text{for all}~ x \in \Lambda_L\}$. Then for any $d>2p$, it was proved in \cite{BDZ95,Sakagawa,Kurt} that
\begin{equation}
\lim_{L\to\infty} \frac{\log \mathbb{P}(\Omega^+_{\Lambda_L})}{L^{d-2p} \log L} = -2p \cdot G_{\mathbb{Z}^d}(0,0)
\cdot {\rm Cap}_p([-1,1]^d).
\label{eq:LDZd}
\end{equation}
Here $G_{\mathbb{Z}^d}(x,y) = \langle \mathbbm{1}_{\{x\}},
(-\Delta)^{-p}\mathbbm{1}_{\{y\}}\rangle_{\ell^2(\mathbb{Z}^d)}$ is the Green's function relative to
$(-\Delta)^p$, and ${\rm Cap}_p(K)$ is the capacity of the set $K \subset \mathbb{R}^d$ relative to the $p$-Laplacian energy
on $L^2(\mathbb{R}^d)$,
$$\mathcal{E}_p(f) = \frac{1}{(2d)^p} \int_{\mathbb{R}^d} |(\nabla^p f)(x)|^2\, dx,$$
that is: ${\rm Cap}_p(K) = \inf\left\{\mathcal{E}_p(f): f\in H^p(\mathbb{R}^d),~f\geq 1~\text{on}~ K
\right\}$, where $H^p(\mathbb{R}^d)$ is the Sobolev space
$$\left\{f \in L^2(\mathbb{R}^d): D^\alpha f \in L^2(\mathbb{R}^d)~\text{for each multi-index}~\alpha~\text{with}~|\alpha|\leq p\right\},$$
and $D^\alpha$ stands for the weak partial derivative $D_{x_1}^{\alpha_1} \cdots D_{x_d}^{\alpha_d}$ when $\alpha=(\alpha_1, \cdots, \alpha_d)$.

Mathematically, to impose the hard wall condition means to condition on the event $\Omega^+_{\Lambda_L}$. We shall be interested in the local
sample mean of the field under $\mathbb{P}(\cdot|\Omega^+_{\Lambda_L})$ as $L\to\infty$. For each $L \in
2\mathbb{N}$, $\epsilon>0$, and $x\in \Lambda_L$, let $\Lambda_{L,\epsilon}(x) = \{z \in \Lambda_L: \max_{1\leq i
\leq d} |z_i-x_i| \leq \epsilon L\}$ be the $\epsilon L$-cubic neighborhood of $x$, and $\bar\varphi_{L,\epsilon}(x)
= |\Lambda_{L,\epsilon}(x)|^{-1} \sum_{z\in \Lambda_{L,\epsilon}(x)}\varphi_z$ be the average of $\varphi$ over
$\Lambda_{L,\epsilon}(x)$. Then in the case $d>2p$, one has that for any $\epsilon>0$ and $\eta>0$,
\begin{equation}
\lim_{L\to\infty} \sup_{\substack{x \in \Lambda_L\\ \Lambda_{L,\epsilon}(x) \subset \Lambda_L}} \mathbb{P}\left(
\left|\frac{\bar\varphi_{L,\epsilon}(x)}{\sqrt{\log L}} - \sqrt{4p \cdot G_{\mathbb{Z}^d}(0,0)} \right| \geq \eta ~\bigg|~
\Omega^+_{\Lambda_L}\right)=0,
\label{eq:heightZd}
\end{equation}
In other words, the local sample mean of the free field is pushed to a height proportional to $\sqrt{\log L}$ above the
hard wall, and the rescaled height converges in probability to a constant function. Equation (\ref{eq:heightZd}) was proved by Sakagawa
\cite{Sakagawa} and Kurt \cite{Kurt} using (\ref{eq:LDZd}) and a series of coarse graining and conditioning
arguments.

For the free field on $\mathbb{Z}^d$, $d\geq 3$, Bolthausen, Deuschel and Zeitouni proved a 
pointwise result \cite{BDZ95}: for any $\epsilon>0$ and $\eta >0$, 
\begin{equation}
\lim_{L\to\infty} \sup_{x \in \Lambda_{L,\epsilon}} \mathbb{P}\left( \left|\frac{\varphi_x}{\sqrt{\log L}} - \sqrt{4 \cdot
G_{\mathbb{Z}^d}(0,0)} \right| \geq \eta ~\bigg|~ \Omega^+_{\Lambda_L}\right)=0,
\end{equation}
where $\Lambda_{L,\epsilon} := \{x \in \Lambda_L: {\rm dist}(x,(\Lambda_L)^c)\geq \epsilon L\}$.

The goal of this paper is to demonstrate entropic repulsion of the free field on a class of infinite fractal graphs, which do not enjoy the same translational invariance as $\mathbb{Z}^d$ does. For concreteness, we will work with graphs associated with high-dimensional \emph{generalized Sierpinski carpets} (\emph{e.g.} the Menger sponge, but \emph{not} the standard 2-dimensional Sierpinski carpet), and obtain the fractal counterparts to Equations (\ref{eq:LDZd}) and (\ref{eq:heightZd}). As a secondary goal, our work may help identify which aspects of the arguments used on $\mathbb{Z}^d$ remain robust in the non-$\mathbb{Z}^d$ setting. There are two main ingredients which enable this generalization:
\begin{itemize}
\item \emph{Spectral convergence} of discrete graphs to a limit object: In the $\mathbb{Z}^d$ setting, the notion already exists that a sequence of suitably scaled discrete Laplacians (on $\mathbb{Z}^d$) converges to the continuum Laplacian (on $\mathbb{R}^d$), \emph{viz.} Brownian motion on $\mathbb{R}^d$ as a scaling limit of simple random walks on $\mathbb{Z}^d$. This convergence is in the \emph{strong operator} topology and hence in the \emph{strong resolvent} topology as well, the latter of which is crucial for the study of the free field. In the case of Sierpinski carpets, we cannot verify strong operator convergence directly, so instead we will show the \emph{Mosco convergence} of the corresponding discrete Dirichlet forms, an equivalent condition to strong resolvent convergence.

\item \emph{Regularity} in the graph structure: As our graphs retain a regular block cell structure, most of the \emph{conditioning} and \emph{coarse graining} arguments employed on $\mathbb{Z}^d$ can be adapted to our setting after some modifications. Translational invariance is not a must.
\end{itemize}

Let us now describe the basics of the Sierpinski carpet, before presenting our main theorems.

\emph{Notation.} Thoughout the article, $c$, $C$ and $C'$ denote positive constants which may change from line to line. The variables upon which the constant depends will be indicated in parentheses (\emph{e.g.} $C(O_1, O_2)$). Specific constants will be indicated with a numeral subscript (\emph{e.g.} $\cc{1.1}$). Meanwhile, $C_c(X)$ denotes the space of continuous functions on $X$ with compact support.

\subsection{Generalized Sierpinski carpets and associated graphs}

\begin{table}
\centering
\renewcommand{\arraystretch}{1.1} 
\begin{tabular}{ccc}
$\mathbb{Z}^d$, $d\geq 3$ & Infinite graph & Transient GSC graph\\ \hline
$[-\frac{L}{2},\frac{L}{2}]^d\cap \mathbb{Z}^d$ & Approximating subgraph (``box'') & $\mathcal{G}_N=(V_N,\sim)$\\
$L$ & Side length of box & $\ell_F^N$ \\
$\asymp L^d$ & Volume of box & $\asymp m_F^N$ \\
$\asymp L^2$ & Expected escape time of random walk from box & $\asymp t_F^N$ \\
$\asymp L^{2-d} \in (0,1)$ & Resistance across opposite faces of box & $\asymp \rho_F^N = (t_F/m_F)^N \in (0,1)$
\end{tabular}
\caption{Comparison of relevant parameters on $\mathbb{Z}^d$ and on the Sierpinski carpet graph.}
\label{tab:params}
\end{table}

\subsubsection{Construction of the fractal} \label{sec:fractal}

Let $F_0:=[0,1]^d$ be the unit cube in $\mathbb{R}^d$, $d\geq 2$, and fix an $\ell_F \in \mathbb{N}$, $\ell_F
\geq 3$. For $N\in \mathbb{Z}$, let $\mathcal{Q}_N$ be the collection of closed cubes of side $\ell_F^{-N}$ with
vertices in $\ell_F^{-N} \mathbb{Z}^d$. For $A \subset \mathbb{R}^d$, let $\mathcal{Q}_N(A) = \{Q \in
\mathcal{Q}_N: {\rm int}(Q)\cap A \neq \emptyset\}$. Denote by $\Psi_Q$ the orientation-preserving affine map
which maps $F_0$ to $Q \in \mathcal{Q}_N$.

We now introduce a decreasing sequence $(F_N)_N$ of closed subsets of $F_0$ as follows. Fix $m_F \in \mathbb{N}$,
$1\leq m_F < \ell_F^d$, and let $F_1$ be the union of $m_F$ distinct elements of $\mathcal{Q}_1(F_0)$. Then by
induction we put
$$F_{N+1} = \bigcup_{Q\in \mathcal{Q}_N(F_N)} \Psi_Q(F_1) = \bigcup_{Q \in \mathcal{Q}_1(F_1)}\Psi_Q(F_N)~,~N\geq
1.$$
It is a standard argument to show that $F= \bigcap_{N=0}^\infty F_N$ is the unique fixed point of the iterated
function system of contractions $\{\Psi_Q\}_{Q\in \mathcal{Q}_1(F_1)}$. Moreover, $F$ has Hausdorff dimension
$d_h(F)=\log m_F/\log \ell_F$.

We say that $F$ is a \emph{generalized Sierpinski carpet (GSC)} if and only if $F_1$ satisfies the following four
conditions:
\begin{enumerate}[label={(GSC\arabic*)},nolistsep]
{\indentitem \item (Symmetry) $F_1$ is preserved under the isometries of the unit cube. \label{cond:GSC1}}
{\indentitem \item (Connectedness) $F_1$ is connected. \label{cond:GSC2}}
{\indentitem \item (Non-diagonality) Let $m\geq 1$ and $B\subset F_0$ be a cube of side length $2\ell_F^{-m}$,
which is the union of $2^d$ distinct elements of $\mathcal{Q}_m$. Then if ${\rm int}(F_1\cap B)$ is non-empty, it
is connected. \label{cond:GSC3} }
{\indentitem \item (Borders included) $F_1$ contains the segment $\{(x_1,0,\cdots,0) \in \mathbb{R}^d: x_1\in
[0,1]\}$. \label{cond:GSC4}}
\end{enumerate}
For equivalent ways of stating the non-diagonality condition \ref{cond:GSC3}, see \cite[\S 2]{KajinoND}. Throughout
the article, we shall refer to $\ell_F$ and $m_F$ as, respectively, the \emph{length scale factor} and the
\emph{mass scale factor} of the carpet $F$.

The stochastic analysis on the Sierpinski carpet is built upon the measure space $(F,\nu)$, where $\nu$ is the self-similar Borel probability measure which assigns mass $m_F^{-N}$ to each $\Psi_Q(F)$, $Q\in \mathcal{Q}_N(F_N)$. Note that $\nu$ is a constant multiple of the $d_h(F)$-dimensional Hausdorff measure on $F$. We will also consider the \emph{unbounded carpet} $F_\infty := \bigcup_{N=0}^\infty\ell_F^N F$, and let $\nu_\infty$ be the $\sigma$-finite self-similar Borel probability measure on $F_\infty$, assigning mass $m_F^N$ to $\ell_F^N F$.

We introduce two other important scale factors associated with Sierpinski carpets. Let $D_N$ be the network of diagonal crosswires obtained by connecting each vertex of a cube $Q\in \mathcal{Q}_N(F_N)$
to the vertex at the center of the cube via a wire of unit resistance. Denote by $\mathcal{R}_{D_N}$ the
resistance across two opposite faces of $D_N$. It was shown in \cite{BB90Resistance,McGillivray} that there
exist $\rho_F \in (0,\infty)$ and positive constants $C(d)$ and $C'(d)$
such that
$$ C \rho_F^N \leq \mathcal{R}_{D_N} \leq C'\rho_F^N.$$
The constant $\rho_F$ is henceforth referred to as the \emph{resistance scale factor} of the carpet $F$. As of
this writing, there is no known exact formula for $\rho_F$: the best estimate, obtained via a resistance shorting
and a cutting argument, is \cite[Proposition 5.1]{BB99}
$$\ell_F^2/m_F \leq  \rho_F \leq 2^{1-d} \ell_F.$$

Next, let $t_F = m_F \rho_F$, which stands for the \emph{time scale factor} of the carpet $F$. The significance of
$t_F$ is due to the fact that the expected
time for a $d$-dimensional Brownian motion to traverse from one face of $\ell_F^N F_N$ to the opposite face scales with $t_F^N$ \cite{BB89, BB90Resistance, BB92, BB99}.

It is often convenient to introduce, respectively, the \emph{Hausdorff}, \emph{walk}, and \emph{spectral} dimensions of $F$:
$$ d_h(F) = \frac{\log m_F}{\log \ell_F},\quad d_w(F) = \frac{\log t_F}{\log \ell_F},\quad d_s(F)= 2\frac{\log
m_F}{\log t_F}.$$
Under the strict inequality $m_F < \ell_F^d$, one has $1\leq d_s(F) < d_h(F) < d$ and $d_w(F)>2$. The latter
inequality implies that diffusion on $F$ (resp. $F_\infty$) is \emph{sub-Gaussian}, in contrast with Gaussian
diffusion which has walk dimension $2$.

\subsubsection{Sierpinski carpet graphs} \label{sec:SCGraph}

\begin{figure}
\centering
\subfigure[]{
\includegraphics{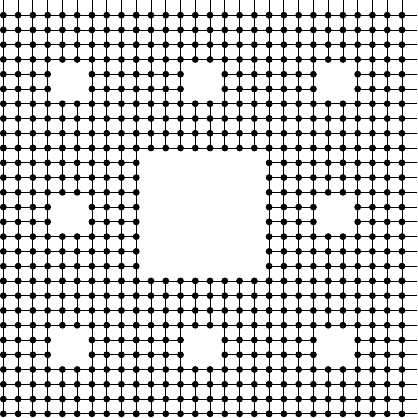}
}
\subfigure[]{
\includegraphics{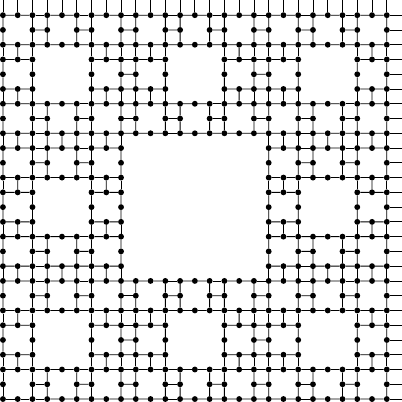}
}
\caption{The 3rd-level approximation of, respectively, the outer Sierpinski carpet graph $\mathcal{G}_\infty$ and the inner Sierpinski carpet graph $\mathcal{I}_\infty$, here shown for the standard 2-dimensional Sierpinski carpet. According to the conventions in the text, when embedded in $(\mathbb{R}_+)^d$, the least vertex of $\mathcal{G}_\infty$ is situated at the origin, while the least vertex of $\mathcal{I}_\infty$ is situated at $(\frac{1}{2},\cdots,\frac{1}{2})$. All edges have Euclidean distance $1$.}
\label{fig:SCGraph}
\end{figure}

For each generalized Sierpinski carpet $F$, we consider two associated graphs. See Figure \ref{fig:SCGraph}.

Let $V_N=\ell_F^N F_N \cap \mathbb{Z}^d$. We introduce the finite graph $\mathcal{G}_N =(V_N,\sim)$, where throughout the paper, the edge relation ``$\sim$'' means that two vertices
$x,x'$ are connected by an edge if and only if their Euclidean distance $\|x-x'\|=1$. Then we put $\mathcal{G}_\infty = (V_\infty,\sim)$ where $V_\infty=\bigcup_{N \in
\mathbb{N}} V_N$, and refer to this as the \emph{outer} Sierpinski carpet graph. Observe that
$\mathcal{G}_\infty$ is a subgraph of $(\mathbb{Z}_+)^d$. In this paper we will study the Gaussian
free field on $\mathcal{G}_\infty$.

Next, let $I_N = \ell_F^N F_N \cap \left(\mathbb{Z}^d +
\left(\frac{1}{2},\frac{1}{2},\cdots,\frac{1}{2}\right)\right)$. We introduce the finite graph $\mathcal{I}_N=(I_N,\sim)$. Then $\mathcal{I}_\infty= (I_\infty, \sim)$, where $I_\infty =
\bigcup_{N\in\mathbb{N}} I_N$, is what we call the \emph{inner} Sierpinski carpet graph. It is easy to see that $|I_N|=m_F^N$, and that there exist constants $\cc{1.1}$ and $\cc{1.2}$, independent of $N$, such that $\cc{1.1} m_F^N \leq |V_N| \leq \cc{1.2} m_F^N$.

For easy reference, we provide  in Table \ref{tab:params} a side-by-side comparison of the relevant parameters on $\mathbb{Z}^d$ and on the Sierpinski carpet graph ($\mathcal{G}_\infty$ or $\mathcal{I}_\infty$). It is known that simple random walk on the latter is transient if and only if $\rho_F<1$ \cite{BBSCGraph, McGillivray}. 

\subsubsection{Dirichlet forms}

Consider the quadratic form
\begin{equation}
\mathcal{E}_{\mathbb{R}^d}(f_1,f_2) = \int_{\mathbb{R}^d} (\nabla f_1)(x) \cdot (\nabla f_2)(x) dx,\quad \forall f_1, f_2 \in H^1(\mathbb{R}^d).
\label{eq:DFRd}
\end{equation}
In the language of potential theory \cite{FOT,ChenFukushima}, $(\mathcal{E}_{\mathbb{R}^d},H^1(\mathbb{R}^d))$ is a strongly local, regular, non-zero, conservative Dirichlet form on $L^2(\mathbb{R}^d)$, and hence can be associated to a Hunt process (diffusion) on $\mathbb{R}^d$, which is nothing but the canonical Brownian motion.

In the same spirit, we shall refer to a \emph{Brownian motion} on a Sierpinski carpet $F$ as a Hunt process associated with a local, regular, non-zero, conservative Dirichlet form on $L^2(F,\nu)$ which is invariant with respect to all the \emph{local symmetries} of $F$. (See \cite[Definition 2.15]{BBKT} for the precise definition of this last condition.) Denote by $\mathfrak{E}_c$ the totality of all such Dirichlet forms on $L^2(F,\nu)$. 

It is straightforward to use elements $(\mathcal{E}_c, \mathcal{F}_c)$ of $\mathfrak{E}_c$ to construct the corresponding Dirichlet forms (or Brownian motions) on the infinite carpet $F_\infty$  (\emph{cf.} \cite[\S 4]{KumagaiBesov}). Let $\tilde{F}_N=\ell_F^N F$, and define $\sigma_N : \{\text{functions on}~\tilde{F}_N\} \to \{\text{functions on}~F\}$ by $(\sigma_N f)(x) = f(\ell_F^N x)$. Put $\tilde{\mathcal{F}}_N = \sigma_{-N}\mathcal{F}_c$ and $\tilde{\mathcal{E}}_N(f,g) = \rho_F^{-N} \mathcal{E}_c(\sigma_N f, \sigma_N g)$ for all $f,g \in \tilde{\mathcal{F}}_N$. It is verified in \cite[(4.7)]{KumagaiBesov} that $\tilde{\mathcal{E}}_{N-1}(\left. f\right|_{\tilde{F}_{N-1}}, \left. f\right|_{\tilde{F}_{N-1}} ) \leq \tilde{\mathcal{E}}_N(f,f)$ for all $f \in \tilde{\mathcal{F}}_N$, which implies that the limit $\lim_{N\to\infty} \tilde{\mathcal{E}}_N\left(\left. f\right|_{\tilde{F}_N}, \left.f\right|_{\tilde{F}_N}\right)$ exists in $\mathbb{R} \cup \{+\infty\}$. Define
\begin{eqnarray*}
\mathcal{D}_\infty &=& \{f\in C_c(F_\infty) : \left.f\right|_{\tilde{F}_N} \in \tilde{\mathcal{F}}_N~\forall N\in \mathbb{N}, ~\lim_{N\to\infty}\tilde{\mathcal{E}}_N(\left.f\right|_{\tilde{F}_N}, \left.f\right|_{\tilde{F}_N})<\infty\}.\\
\mathcal{E}_\infty(f,g) &=& \lim_{N\to\infty} \tilde{\mathcal{E}}_N(\left.f\right|_{\tilde{F}_N}, \left.f\right|_{\tilde{F}_N}),\quad \forall f,g \in \mathcal{D}_\infty.
\end{eqnarray*}
Then $(\mathcal{E}_\infty, \mathcal{D}_\infty)$ is a closable Markovian form on $L^2(F_\infty,\nu_\infty)$, and has a smallest closed extension $(\mathcal{E}_\infty, \mathcal{F}_\infty)$, which is local, regular, non-zero, conservative, and invariant with respect to all the local symmetries. We denote the totality of all such Dirichlet forms on $L^2(F_\infty,\nu_\infty)$ by $\mathfrak{E}$.

We should mention that the construction of a Brownian motion on a Sierpinski carpet has a long and rich history. In \cite{BB89,BB92,BB99}, Barlow and Bass studied reflecting Brownian motions $W^N_t$ on the sequence of approximating carpets $F_N$, and showed that the laws of the sped-up Brownian motions $(X^N_t)_N = (W^N_{a_N t})_N$ are tight, where the numbers $a_N$ satisfy $c (t_F/\ell_F^2)^N \leq a_N \leq C (t_F/\ell_F^2)^N$. Thus $(X^N_t)_N$ admits subsequential limits, which Barlow and Bass referred to as ``Brownian motions'' on the Sierpinski carpet $F= \bigcap_N F_N$. (Brownian motions on the infinite carpet $F_\infty$ are obtained similarly.) In \cite{KusuokaZhou}, Kusuoka and Zhou provided a different construction based on Dirichlet forms on approximating graphs $\mathcal{I}_N$ (see \S \ref{sec:KZ} for a brief discussion). For nearly two decades it was unclear whether the two approaches gave rise to the same \emph{unique} Brownian motion. A recent seminal work \cite{BBKT} showed that up to constant multiples, $\mathfrak{E}_c$ contains only one element. In particular, the Dirichlet forms associated with the Barlow-Bass construction and the Kusuoka-Zhou construction both belong to $\mathfrak{E}_c$, and therefore the two diffusions are the same up to deterministic time change.

In light of this uniqueness result, all of our ensuing statements which contain the phrase ``there exists $\mathcal{E}\in \mathfrak{E}$'' or ``for some $\mathcal{E} \in \mathfrak{E}$'' may be replaced by ``for any $\mathcal{E} \in \mathfrak{E}$'' without trouble.

\subsection{Main results}

In what follows, $F$ is a generalized Sierpinski carpet which supports \emph{transient} Brownian motion, \emph{i.e.,} $\rho_F<1$, or equivalently, $d_s(F)>2$. This includes the standard $d$-dimensional carpet for $d\geq 3$ (where $F_1$ is obtained by removing a center cube of side $1/3$); the standard $3$-dimensional Menger sponge (where $F_1$ is obtained by removing a cube of side $1/3$ centered on each of the 6 faces of the unit cube, plus the cube of side $1/3$ situated at the body center); and many, \emph{but not all}, carpets in $\mathbb{R}^d$ with $d\geq 3$. Our analysis does \emph{not} apply to any generalized Sierpinski carpet in $\mathbb{R}^2$, whereby $\rho_F>1$ and is hence recurrent.

Let $G_{\mathcal{G}_\infty}: V_\infty \times V_\infty \to \mathbb{R}$ be the Green's function for simple random walk on the outer Sierpinski carpet graph $\mathcal{G}_\infty$ without killing. From basic Markov chain theory, a vertex $x$ is recurrent with respect to the random walk if and only if the expected number of returns to $x$ is infinite:
$$
\mathbb{E}^x\left[\sum_{k=1}^\infty \mathbbm{1}_{\{Y_k=x\}}\right]=\infty.
$$
By this result, if $x$ is transient, then the Green's function $G_{\mathcal{G}_\infty}(x,x)$ is finite:
$$
G_{\mathcal{G}_\infty}(x,x) = \frac{1}{\deg(x)}\mathbb{E}^x\left[\sum_{k=0}^\infty \mathbbm{1}_{\{Y_k=x\}}\right] = \frac{1}{\deg(x)}\left(1+\mathbb{E}^x\left[\sum_{k=1}^\infty  \mathbbm{1}_{\{Y_k=x\}}\right]\right) <\infty.
$$
Moreover, the random walk on $V_\infty$ is irreducible, so $G_{\mathcal{G}_\infty}(x,x) <\infty$ for all $x\in V_\infty$. Let us denote $\overline{G}:=\sup_{x\in V_\infty}G_{\mathcal{G}_\infty}(x,x)$ and $\underline{G} :=
\inf_{x\in V_\infty}G_{\mathcal{G}_\infty}(x,x)$, which are both positive and finite.

We also need the notion of the ($0$-order) capacity of the compact carpet $F$ with respect to a Dirichlet form
$(\mathcal{E},\mathcal{F})\in \mathfrak{E}$ on $L^2(F_\infty, \nu_\infty)$, given by
$$ {\rm Cap}_{\mathcal{E}}(F):= \inf\{\mathcal{E}(f,f): f\in \mathcal{F}\cap C_c(F_\infty), f\geq 1 \text{~a.e.
on}~ F\},$$
See Proposition \ref{prop:capacity} for a more general definition of the capacity, as well as some important
properties.

Let $\mathbb{P}$ be the law of the Gaussian free field on $\mathcal{G}_\infty$ with covariance
$G_{\mathcal{G}_\infty}$, and let $\Omega_{V_N}^+$ denote the entropic repulsion event $\{\varphi_x \geq 0 ~\text{for all}~
x \in V_N\}$. Our first main result identifies the rate of exponential decay for $\mathbb{P}(\Omega_{V_N}^+)$.

\begin{theorem}
There exists a point $x_0 \in V_\infty$ such that for any Dirichlet form $(\mathcal{E},\mathcal{F}) \in
\mathfrak{E}$, there are positive constants $\cc{1.3}$ and $\cc{1.4}$ such that
\begin{equation}
-\cc{1.3}\cdot\overline{G} \cdot{\rm Cap}_{\mathcal{E}}(F) \leq \varliminf_{N\to\infty}
\frac{\log\mathbb{P}(\Omega_{V_N}^+)}{\rho_F^{-N} \log(t_F^N)} \leq \varlimsup_{N\to\infty}
\frac{\log\mathbb{P}(\Omega_{V_N}^+)}{\rho_F^{-N} \log(t_F^N)} \leq -\cc{1.4} \cdot
G_{\mathcal{G}_\infty}(x_0,x_0) \cdot{\rm Cap}_{\mathcal{E}}(F).
\end{equation}
\label{thm:LD}
\end{theorem}

The constants $\cc{1.3}$ and $\cc{1.4}$ are attributed to two sources: one coming from comparing the Dirichlet forms on $\mathcal{G}_\infty$ and on $\mathcal{I}_\infty$ (Lemma \ref{lem:Ecomp}), and
the other coming from comparing the (maximal or minimal) cluster point of the sequence of renormalized Dirichlet
forms on $\mathcal{I}_\infty$ with an element of $\mathfrak{E}$ (Theorem
\ref{thm:Greenform}). The constant $\cc{1.4}$ also depends on the constant $\cc{1.2}$ in the growth of the cardinality $|V_N| \leq \cc{1.2} m_F^N$. (For future reference on how the constants are related, $\cc{1.3} = \frac{1}{2} \cc{2.2}^{-1}\cc{2.8} d_s(F) $ and $\cc{1.4}=(\cc{1.2}^2 \cc{2.6})^{-1}$.) Due to the lack of precise control of the constants involved
in these comparisons, \emph{the authors deem it not possible to determine whether $\cc{1.3}$ equals $\cc{1.4}$ by the methods of this paper.}

Notwithstanding the small discrepancy between the lower and upper bounds in Theorem \ref{thm:LD}, we are still able
to give a precise description of entropic repulsion on $\mathcal{G}_\infty$. We shall prove that conditional on
$\Omega^+_{V_N}$, the local sample mean of the free field on $V_N$ is pushed to a height which is proportional to
$\sqrt{N}$, and as $N\to\infty$, the rescaled height converges in probability to a constant.

\begin{theorem}
For any $\epsilon>0$ and $\eta>0$,
\begin{equation}
\lim_{N\to\infty} \sup_{\substack{x\in V_N \\ V_{N,\epsilon}(x) \subset V_N}}
\mathbb{P}\left(\left|\frac{\bar\varphi_{N,\epsilon}(x)}{\sqrt{\log(t_F^N)}}-\sqrt{2\underline{G}}\right| \geq \eta~
\bigg| ~\Omega^+_{V_N}\right)=0,
\label{eq:samplemeanheight}
\end{equation}
where $\displaystyle \bar\varphi_{N,\epsilon}(x) := \frac{1}{|V_{N,\epsilon}(x)|}\sum_{z\in V_{N,\epsilon}(x)}
\varphi_z$ and $\displaystyle V_{N,\epsilon}(x) := \left\{z \in V_N: \max_{1\leq i\leq d}\left|z_i-x_i \right| \leq \epsilon\cdot \ell_F^N\right\}$.
\label{thm:samplemeanheight}
\end{theorem}

We comment that in the case of
$\mathbb{Z}^d$, which has full translational invariance, one can replace $G_{\mathbb{Z}^d}(0,0)$ by
$G_{\mathbb{Z}^d}(x,x)$ for any $x\in \mathbb{Z}^d$. On the other hand, in the case of the Sierpinski carpet
graph, we have no explicit information about how the on-diagonal Green's function $G_{\mathcal{G}_\infty}(x,x)$
varies with $x \in V_\infty$. Nevertheless, our result says that $\sqrt{2\underline{G} \log t_F} \sqrt{N}$, where
$\underline{G} = \inf_{x\in V_\infty}G_{\mathcal{G}_\infty}(x,x)$, sets the leading-order asymptotic height for the
free field above the hard wall on $V_N$. A sketch of the arguments leading to this result will appear at the
beginning of \S\ref{subsec:heightub}.

An interesting open problem is to investigate the maxima and entropic repulsion of the free field on \emph{any}
graph at the critical spectral dimension $d_s=2$. Due to the logarithmic correlation of the random field, one
needs to carry out a multiscale analysis which differs significantly from the approach for higher dimensions. For
some recent progress in characterizing the fine properties of the free field on $\mathbb{Z}^2$, see, for example,
\cite{BDG01,Daviaud,BDZ11,BramsonZeitouni}. Unfortunately, we are not in a position to discuss the analogous
problem in the setting of deterministic self-similar generalized Sierpinski carpets, since there is no concrete
example of such a carpet with $d_s=2$. A perhaps more promising avenue is to analyze the free field on a random
homogeneous Sierpinski carpet, for which there exists a protocol to construct one with arbitrary spectral
dimension \cite{HKKZ}. Other possible candidates include the supercritical percolation cluster and the random conductance model on $\mathbb{Z}^2$, whereupon the random walk has been to proved to satisfy an
invariance principle \cite{DMFGW84, SidoraviciusSznitman, BergerBiskup, MathieuPiatnitski, Mathieu,
BarlowDeuschel, ABDH13}.

The rest of the article is organized as follows. In Section \ref{sec:DFGF} we recapitulate the construction of
local regular Dirichlet forms on Sierpinski carpets via graphical approximations (the Kusuoka-Zhou construction),
and prove the convergence of the discrete Green forms, both on $\mathcal{I}_\infty$ and on $\mathcal{G}_\infty$, to a continuum Green form on $F_\infty$. We then proceed to
prove Theorem \ref{thm:LD} in Section \ref{sec:LD}, and Theorem \ref{thm:samplemeanheight} in Section
\ref{sec:height}. Some background facts from the theory of Dirichlet forms and of Mosco convergence are collected in the Appendix.

\vskip 10pt

\noindent \textbf{Acknowledgements.} 
This work was initiated while the first named author was a visiting researcher under the Hausdorff Trimester Program
\emph{``Mathematical challenges of materials science and condensed matter physics,''} hosted by the Hausdorff Research
Institute for Mathematics in Bonn, Germany. He would like to thank Codina Cotar and Noemi Kurt for fruitful conversations
which inspired this work, and all the program participants for their camaraderie. During the ensuing preparation of the manuscript at Cornell, both authors have benefited from valuable discussions with
Laurent Saloff-Coste, Benjamin Steinhurst, and especially Robert Strichartz, to whom they are grateful for generous advising. Lastly, we thank the referee for providing excellent advice on improving an earlier version of this paper.

\section{Dirichlet forms and Green forms on Sierpinski carpets} \label{sec:DFGF}

The goal of this section is to establish several convergence results of resolvents (or ``Green forms'') on Sierpinski carpet graphs, which will enable us to prove Theorems \ref{thm:LD} and \ref{thm:samplemeanheight}. Here the route is convoluted, as we have to leverage several known but disparate results to establish something which is believed to be true, but not spelled out prominently in the literature. To give an example, even the resolvent convergence (\emph{i.e.} Mosco convergence) is not straightforward to prove on the Sierpinski carpet (graph), although we will present an argument in this section. The main take-away results are Theorem \ref{thm:Greenform} (convergence of the Green forms) and Lemma \ref{lem:VNConv} (consequences of Theorem \ref{thm:Greenform}, suitably specialized to the graph setting); only
Lemma \ref{lem:VNConv} will play a role in subsequent sections.

\emph{Notation.} If $(X,m)$ denotes a measure space, then $\langle
f,\mu\rangle_X$ stands for $\int_X fd\mu$, pairing a function $f$ on $X$ with a Borel measure $\mu$.

\subsection{Kusuoka-Zhou construction of Dirichlet forms} \label{sec:KZ}

Let $F$ be a generalized Sierpinski carpet, and $\mathcal{I}_\infty = (I_\infty,\sim)$ be the inner Sierpinski carpet graph introduced in \S \ref{sec:SCGraph}. For each $N\in\mathbb{N}$ and each $w\in I_\infty$, let $\Psi^{(N)}_w$ be the closed cube of side $\ell_F^{-N}$ centered at $\ell_F^{-N} w$. We define the \emph{mean-value operator} $\tilde{P}_N: L^1(F_\infty,\nu_\infty) \to C(I_\infty;\mathbb{R})$ by
$$ (\tilde{P}_Nf)(w) = \frac{1}{\nu_\infty\left(\Psi^{(N)}_w \cap F_\infty\right)} \int_{\Psi^{(N)}_w \cap F_\infty} f(y) \nu_\infty(dy),$$
Similarly, if $\mu_\infty$ is a Radon measure on $F_\infty$ such that $\mu_\infty \ll \nu_\infty$, then define $\tilde P_N \mu_\infty = \left(\tilde P_N\frac{d\mu_\infty}{d\nu_\infty}\right) \nu_N$, where $\nu_N = \frac{1}{m_F^N}
\mathbbm{1}_{I_\infty}$ is a self-similar measure on $I_\infty$.

As is customary, we define the discrete Dirichlet form on the graph $\mathcal{I}_\infty$ by
$$ E_{\mathcal{I}_\infty}(f_1,f_2) = \frac{1}{2} \sum_{\substack{w,w'\in I_\infty\\w\sim w'}}(f_1(w)-
f_1(w'))(f_2(w)-f_2(w'))$$
for all $f_1, f_2$ in the natural domain $\mathcal{D}(E_{\mathcal{I}_\infty}) = \{f \in \ell^2(I_\infty) :
E_{\mathcal{I}_\infty}(f,f)<\infty\}$. Furthermore, let $\mathcal{E}_N^\mathcal{I}= \rho_F^N E_{\mathcal{I}_\infty}$ be the \emph{renormalized} Dirichlet form, where $\rho_F \in (0,\infty)$ is the resistance scale factor identified in \S \ref{sec:fractal}. 

Let $\displaystyle \mathcal{F}_0 := \left\{f\in L^2(F_\infty,\nu_\infty): \sup_N \mathcal{E}_N^\mathcal{I}(\tilde P_N f, \tilde P_N f) <\infty\right\}$. The following convergence result for $(\mathcal{E}_N^\mathcal{I})_N$ is originally due to Kusuoka and Zhou \cite[Proposition 5.2 \& Theorem 5.4]{KusuokaZhou}, and later generalized in \cite[Lemma 4.1 \& Theorem 4.3]{HKKZ}.

\begin{proposition}
\begin{enumerate}[label={(\roman*)}]
\item There exists a constant $\cc{2.1}$ such that for all $N,M\geq 1$ and all $f\in \mathcal{F}_0$, 
$$\mathcal{E}_N^\mathcal{I}(\tilde P_N f, \tilde P_N f) \leq \cc{2.1} \mathcal{E}_{N+M}^\mathcal{I}(\tilde P_{N+M} f,\tilde P_{N+M} f).$$
\item There exists $(\mathcal{E},\mathcal{F}_0)\in \mathfrak{E}$ and positive constants $\cc{2.2}$ and $\cc{2.3}$ such that for all $f\in \mathcal{F}_0$,
\begin{equation}
\cc{2.2} \sup_N \mathcal{E}_N^\mathcal{I}(\tilde P_N f,\tilde P_N f) \leq \mathcal{E}(f,f) \leq \cc{2.3} \varliminf_{N\to\infty} \mathcal{E}_N^\mathcal{I}(\tilde P_N f,\tilde P_N f).
\label{eq:KZ}
\end{equation}
\end{enumerate}
\label{prop:KZ}
\end{proposition}

\begin{remark}
In their original work \cite{KusuokaZhou}, Kusuoka and Zhou identified a family of Dirichlet forms, denoted $\mathcal{D}ch$, which are associated with cluster points of the sequence of suitably rescaled Markov processes on $\mathcal{I}_N$. Then they proved (\ref{eq:KZ}) for any $\mathcal{E}\in \mathcal{D}ch$, and showed that $(\mathcal{E}, \mathcal{F}_0)$ is a local regular Dirichlet form. Note that $\mathcal{D}ch \subset \mathfrak{E}$ by virtue of \cite[Theorem 3.2]{BBKT}.
\end{remark}

\begin{remark}
We emphasize that Proposition \ref{prop:KZ} does not imply that the limit points of $\left(\mathcal{E}_N^\mathcal{I}\right)_N$ (in either the pointwise, $\Gamma$-, or the Mosco topology) belong to $\mathfrak{E}$. 
Rather, each of them is \emph{comparable} to any element of $\mathfrak{E}$, 
in the sense that for any limit point $\bar{\mathcal{E}}$ and any $\mathcal{E} \in \mathfrak{E}$, 
there exist positive constants $c$ and $C$ such that $c\mathcal{E}(f,f)\leq\bar{\mathcal{E}}(f,f) \leq C\mathcal{E}(f,f)$ for all $f\in \mathcal{F}_0$. See Theorem \ref{thm:Moscolimit} below.
\end{remark}

\subsection{Convergence of discrete Green forms}

In this subsection we shall consider Dirichlet forms on a class of smooth measures (instead of functions), and derive a convergence result similar to Proposition \ref{prop:KZ}. From now on let $\mathcal{M}_+(F_\infty)$ be the family of all nonnegative finite Borel measures on $F_\infty$, and let
$$ 
\mathcal{M}_{0,{\rm ac}}^{(0)}(F_\infty) = \left\{ \mu \in \mathcal{M}_+(F_\infty): \mu \ll \nu,~\frac{d\mu}{d\nu} \in \mathcal{F}_0 \right\}.
$$
In what follows we need a fact from the theory of Dirichlet forms (see the Appendix for a brief summary, and more details in \cite{FOT, ChenFukushima}): a transient Dirichlet form $(\mathcal{E},\mathcal{F})$ on $L^2(X,m)$ has a corresponding $0$-order potential operator $U$ such that $\mathcal{E}(U\mu,h) = \langle h,\mu\rangle_X$ for every ``smooth'' measure $\mu$ on $X$ and every $h\in \mathcal{F}$.

Let $G_{\mathcal{I}_N}: I_N \times I_N \to \mathbb{R}$ be the
Green's function for simple random walk on $\mathcal{I}_\infty$ killed upon exiting $\mathcal{I}_N$. By the
reproducing property of Green's function,
$E_{\mathcal{I}_\infty}(G_{\mathcal{I}_N}(w,\cdot),h) = h(w)$ for all $h\in \mathcal{D}(E_{\mathcal{I}_\infty})$
with ${\rm supp}(h) \subset I_N$. Therefore, denoting by $U_N^\mathcal{I}$ the $0$-order potential operator associated
with $\mathcal{E}_N^\mathcal{I}$, we have
$$
\mathcal{E}_N^\mathcal{I}(U_N^\mathcal{I} \mu, h) = \langle h, \mu\rangle_{\mathcal{I}_N} = \frac{1}{m_F^N}
\sum_{w\in I_N}h(w) \frac{d\mu}{d\nu_N}(w)= \mathcal{E}_N^\mathcal{I}\left( \rho_F^{-N}\frac{1}{m_F^N}\sum_{w\in I_N}
G_{\mathcal{I}_N}(\cdot,w)\frac{d\mu}{d\nu_N}(w),h\right)
$$
for all $h\in\mathcal{D}(E_{\mathcal{I}_\infty})$ with ${\rm supp}(h) \subset I_N$, and all nonnegative measures $\mu$ with support in $I_N$. It follows that
\begin{eqnarray}
 \nonumber \mathcal{E}_N^\mathcal{I}(U_N^\mathcal{I} \mu, U_N^\mathcal{I} \mu) &=&
\mathcal{E}_N^\mathcal{I}\left(\rho_F^{-N}\frac{1}{m_F^N}\sum_{w\in I_N}
G_{\mathcal{I}_N}(\cdot,w)\frac{d\mu}{d\nu_N}(w),\rho_F^{-N}\frac{1}{m_F^N}\sum_{w'\in I_N}
G_{\mathcal{I}_N}(\cdot,w')\frac{d\mu}{d\nu_N}(w') \right)\\ 
&=& \rho_F^{-N} \frac{1}{m_F^{2N}}\sum_{w,w' \in I_N} G_{\mathcal{I}_N}(w,w')\frac{d\mu}{d\nu_N}(w)
\frac{d\mu}{d\nu_N}(w') \label{eq:Greenformdef}
\end{eqnarray}
for all such measures $\mu$. The expression in (\ref{eq:Greenformdef}) is what we shall call the
\textbf{Green form} corresponding to the Dirichlet form $\mathcal{E}_N^\mathcal{I}$. It has a kernel given by the
(renormalized) Green's function $\rho_F^{-N} G_{\mathcal{I}_N}$, whence the name.

Our first main result of this section describes the convergence of the discrete Green forms.

\begin{theorem}
There exist $(\mathcal{E},\mathcal{F}) \in \mathfrak{E}$ and constants $\cc{2.4}(\mathcal{E}), \cc{2.5}(\mathcal{E})$ such that
\begin{eqnarray}
\nonumber \cc{2.4} \mathcal{E}(U\mu, U\mu) &\leq& \varliminf_{N\to\infty}\mathcal{E}_N^\mathcal{I}\left(U_N^\mathcal{I}
\tilde P_N \mu, U_N^\mathcal{I} \tilde P_N \mu\right) \\ 
 &\leq&
\varlimsup_{N\to\infty}\mathcal{E}_N^\mathcal{I}\left(U_N^\mathcal{I} \tilde P_N \mu, U_N^\mathcal{I} \tilde P_N
\mu\right) ~ \leq ~ \cc{2.5} \mathcal{E}(U\mu, U\mu) \label{eq:Greenforms}
\end{eqnarray}
for all $\mu \in \mathcal{M}_{0,{\rm ac}}^{(0)}(F_\infty)$, where $U$ is the $0$-order potential operator associated with $\mathcal{E}$.
\label{thm:Greenform}
\end{theorem}

While it is reasonable to expect that Theorem \ref{thm:Greenform} follows from Proposition \ref{prop:KZ}, the connection is \emph{not} immediate. To fill in the necessary gap, we shall establish the (subsequential) \emph{Mosco convergence} of the discrete Dirichlet forms $(\mathcal{E}_N^\mathcal{I})_N$, which is an equivalent condition to (subsequential) \emph{strong resolvent convergence} in $L^2$ \cite{DalMaso,Mosco94}. This then implies (subsequential) \emph{strong resolvent convergence in the energy norm} (Lemma \ref{lem:resolventconv}). In addition, we will show that each Mosco limit point of $(\mathcal{E}_N^\mathcal{I})_N$ is comparable to some (and hence any) element of $\mathfrak{E}$, which finally leads to Theorem \ref{thm:Greenform}. 

Our analysis of the Mosco limits draws from similar analysis of the $\Gamma$-limits by Sturm \cite{SturmDiffusion} and by Kumagai and Sturm \cite{KumagaiSturm}, whose goal was to construct diffusions on an arbitrary metric measure space via $\Gamma$-limits. For the reader's convenience, we have collected some general notions about $\Gamma$-convergence and Mosco convergence in the Appendix.

\begin{theorem}
Let $\displaystyle \mathcal{E}^*(f,f) = \varlimsup_{N\to\infty} \mathcal{E}_N^\mathcal{I}(\tilde P_N f, \tilde P_N f)$ and $\mathcal{F}^\ast := \{f \in C_c(F_\infty): \mathcal{E}^\ast(f,f)<\infty \}$. 
\begin{enumerate}[label={(\roman*)}]
\item $\mathcal{F}^\ast$ is a Lipschitz space and is dense in $C_c(F_\infty)$.
\item There exists a subsequence $(r_N)_N \subset \mathbb{N}$ along which the Mosco limit $\displaystyle \mathcal{E}_M := M\text{-}\lim_{N\to\infty} \mathcal{E}_{r_N}^\mathcal{I}$ exists. Moreover, given any two Mosco limit points $\mathcal{E}_M$ and $\mathcal{E}_M'$, there exist constants $C$ and $C'$ such that $C\mathcal{E}_M(f,f) \leq \mathcal{E}_M'(f,f) \leq C'\mathcal{E}_M(f,f)$ for all $f\in \mathcal{F}^*$.
\end{enumerate}
For each Mosco limit point $\mathcal{E}_M$ of $(\mathcal{E}_N^\mathcal{I})_N$:
\begin{enumerate}[label={(\roman*)}]
\setcounter{enumi}{2}
\item $(\mathcal{E}_M, \mathcal{F}^\ast)$ is a closable symmetric Markovian form on $L^2(F_\infty, \nu_\infty)$, and can be extended to a regular Dirichlet form $(\bar{\mathcal{E}}, \bar{\mathcal{F}})$ on $L^2(F_\infty,\nu_\infty)$ with core $\mathcal{F}^\ast$.
\item There exist $\mathcal{E} \in \mathfrak{E}$ and constants $C$ and $C'$ such that $C \mathcal{E} \leq \bar{\mathcal{E}} \leq C'\mathcal{E}$.
\end{enumerate}
\label{thm:Moscolimit}
\end{theorem}

\begin{remark}
We do not, nor need not, assert that $\bar{\mathcal{E}} \in \mathfrak{E}$. One reason is that we cannot check whether $\bar{\mathcal{E}}$ is (strongly) local. In \cite[Theorem 4.4.1]{Mosco94} a sufficient condition was given for a sequence of (strongly) local regular Dirichlet forms on $L^2(X,m)$ to $\Gamma$-converge to a (strongly) local regular Dirichlet form. The condition states that the sequence of \emph{energy measures} be bounded and absolutely continuous with respect to the reference measure $m$ on $X$. However, in the fractal setting, the energy measure and the self-similar measure on the limiting fractal set are mutually singular \cite{HinoSingularity} (see also \cite{Kusuoka89,BenBassat} for the version of this statement on post-critically finite fractals, such as the Sierpinski gasket).

On the other hand, since $(\bar{\mathcal{E}},\bar{\mathcal{F}})$ is a regular Dirichlet form, one has $\bar{\mathcal{E}}=\bar{\mathcal{E}}^{(c)} + \bar{\mathcal{E}}^{(j)} + \bar{\mathcal{E}}^{(k)}$ by the Beurling-Deny formula (\emph{cf.} \cite[Section 3.2]{FOT}), where $\bar{\mathcal{E}}^{(c)}$, $\bar{\mathcal{E}}^{(j)}$, $\bar{\mathcal{E}}^{(k)}$ stands respectively for the diffusion, jump, and killing part of $\bar{\mathcal{E}}$. In particular, $\bar{\mathcal{E}}^{(c)}$ is strongly local, and assuming Proposition \ref{prop:KZ}, one can show that $\bar{\mathcal{E}}$ and $\bar{\mathcal{E}}^{(c)}$ are comparable (\emph{cf.} \cite[\S 2]{KumagaiSturm}).
\end{remark}

\begin{proof}[Proof of Theorem \ref{thm:Moscolimit}]
(i): By \cite[Theorem 1.3]{BB99}, there is a $(\mathcal{E},\mathcal{F}) \in \mathfrak{E}$ associated with a diffusion on $F_\infty$ whose heat kernel satisfies the following estimate: there exist $\cc{1},\cc{2},\cc{3}$ and $\cc{4}$ such that for all $x,y \in F_\infty$ and $t>0$,
$$
\cc{1}t^{-d_h/d_w}\exp\left(-\cc{2}\left(\frac{\|x-y\|^{d_w}}{t}\right)^{\frac{1}{d_w-1}}\right) \leq p_t(x,y) \leq \cc{3}t^{-d_h/d_w}\exp\left(-\cc{4}\left(\frac{\|x-y\|^{d_w}}{t}\right)^{\frac{1}{d_w-1}}\right).
$$
Then by \cite[Theorem 4.1]{KumagaiSturm}, $\mathcal{F} = {\rm Lip}_{\nu_\infty}(\frac{d_w}{2},2,\infty)(F_\infty)$, the Besov-Lipschitz space which consists of all $f\in L^2(F_\infty,\nu_\infty)$ such that
$$
\sup_{N\in\mathbb{N}\cup\{0\}} \alpha^{N(d_w+d_h)}\int\int_{|x-y|<c_0\alpha^{-N}}|f(x)-f(y)|^2 d\nu_{\infty}(x)d\nu_{\infty}(y) < \infty
$$
for some $\alpha>1$ and $c>0$. (See also \cite{KumagaiBesov} for an earlier derivation.) Using the inequality (\ref{eq:KZ}), we deduce that $\mathcal{F}^*=\mathcal{F}$ is a Lipschitz space and hence is dense in $C_c(F)$. Note that this is Assumption (B1) in \cite[\S 3]{KumagaiSturm}.

(ii): By \cite[Proposition 8.10]{DalMaso}, it suffices to show that $L^2(F_\infty,\nu_\infty)$ has a separable dual (clear), and that there exists a function $\xi:L^2(F_\infty,\nu_\infty) \to \mathbb{R}\cup \{+\infty\}$ such that the following two properties hold:
\begin{itemize}
\item $\displaystyle \lim_{\|f\|_{L^2}\to+\infty}\xi(f) = +\infty$
\item $\mathcal{E}_N^\mathcal{I}(\tilde P_N f, \tilde P_N f) \geq \xi(f)$ for all $f\in L^2(F_\infty,\nu_\infty)$ and all sufficiently large $N$.
\end{itemize}
We claim that it is enough to take $\xi(f) = C\mathcal{E}(f,f)$ for some $(\mathcal{E},{\rm dom}(\mathcal{E})) \in \mathfrak{E}$ and constant $C$. To see that the first property holds, note that ${\rm dom}(\mathcal{E})$ is densely defined on $L^2(F_\infty,\nu_\infty)$, so any function $f$ with infinite $L^2$-norm has infinite $\mathcal{E}$-norm. The second property follows from the right inequality in (\ref{eq:KZ}).

As a consequence, $L^2(F_\infty,\nu_\infty)$ equipped with either the strong or the weak topology has a countable base, so by \cite[Theorem 8.5]{DalMaso}, any (sub)sequence of $(\mathcal{E}_N^\mathcal{I})_N$ contains a Mosco- (resp. weak $\Gamma$-) convergent (sub)subsequence. Moreover, $\mathcal{E}_M(f,f) \leq \mathcal{E}^\ast(f,f)<\infty$ for all $f\in \mathcal{F}^\ast$ and for any Mosco limit point $\mathcal{E}_M$.

(iii): From Part (i) we have closability and regularity. Symmetry and the Markovian property can be verified easily.

(iv): Take any Mosco convergent subsequence $(\mathcal{E}_{r_N}^\mathcal{I})_N$ and denote its limit point (resp. the smallest closed extension thereof) by $\mathcal{E}_M$ (resp. $\bar{\mathcal{E}}$). By Proposition \ref{prop:KZ}(i), for all $N\geq 1$ we have
$$
\mathcal{E}_N^\mathcal{I}(\tilde P_N f, \tilde P_N f) \leq \cc{2.1} \cdot \Gamma_w\text{-}\varliminf_{N'\to\infty} \mathcal{E}_{r_{N'}}^\mathcal{I}(\tilde P_{r_{N'}} f, \tilde P_{r_{N'}} f) = \cc{2.1} \cdot \bar{\mathcal{E}}(f,f).
$$
Taking $\sup_N$ on both sides and using Proposition \ref{prop:KZ}(ii) yields $\bar{\mathcal{E}}(f,f) \geq C\mathcal{E}^\ast(f,f) \geq C' \mathcal{E}(f,f)$ for some $\mathcal{E}\in \mathfrak{E}$.
Meanwhile $\bar{\mathcal{E}}(f,f) \leq \mathcal{E}^\ast(f,f) \leq C''\mathcal{E}(f,f)$ via Proposition \ref{prop:KZ}(ii) again. This completes the proof.\end{proof}

\begin{proof}[Proof of Theorem \ref{thm:Greenform}]

By Proposition \ref{prop:MC} and Lemma \ref{lem:resolventconv}, given each Mosco convergent subsequence $(\mathcal{E}_{r_N}^\mathcal{I})_N$ with limit $\mathcal{E}_M$ and its smallest closed extension $\bar{\mathcal{E}}$, we have that for all $\mu \in \mathcal{M}_{0,{\rm ac}}^{(0)}(F_\infty)$,
$$\lim_{N\to\infty} \mathcal{E}_{r_N}^\mathcal{I}\left(U_{r_N}^\mathcal{I} \tilde P_{r_N} \mu, U_{r_N}^\mathcal{I} \tilde P_{r_N} \mu\right) = \bar{\mathcal{E}}\left(\bar{U}\mu, \bar{U}\mu\right),$$
where $\tilde P_{r_N} \mu = \left(\tilde P_{r_N} \frac{d\mu}{d\nu}\right) \nu_{r_N}$.
According to Theorem \ref{thm:Moscolimit}(iv), there exist $(\mathcal{E},\mathcal{F}) \in \mathfrak{E}$ and constants $C(\mathcal{E}),C'(\mathcal{E})$ such that for all $f\in \mathcal{F}$,
$$ C\mathcal{E}(f,f) \leq \bar{\mathcal{E}}(f,f) \leq C' \mathcal{E}(f,f).$$
If $f$ is a $0$-order potential relative to $\mathcal{E}$, then we can write $f=U\mu$ for some $\mu \in S_0^{(0)}$. Using the identity
$$
\mathcal{E}(U\mu, h)=\int \widetilde{h}d\mu = \bar{\mathcal{E}}(\bar{U}\mu, h)\qquad \text{for all}~\mu \in S_0^{(0)}~\text{and}~h\in \mathcal{F}_e,
$$
(where $\widetilde{h}$ is the quasi-continuous modification of $h$; see the Appendix), we then apply the Cauchy-Schwarz inequality to $\mathcal{E}$, and deduce that for all $\mu \in \mathcal{M}_{0,{\rm ac}}^{(0)}(F_\infty)$,
$$
[\bar{\mathcal{E}}(\bar{U}\mu, \bar{U}\mu)]^2 = [\mathcal{E}(U\mu,\bar{U}\mu)]^2 \leq \mathcal{E}(U\mu,U\mu) \mathcal{E}(\bar{U}\mu, \bar{U}\mu) \leq C^{-1} \mathcal{E}(U\mu, U\mu) \bar{\mathcal{E}}(\bar{U}\mu, \bar{U}\mu),
$$
or $\bar{\mathcal{E}}(\bar{U}\mu,\bar{U}\mu) \leq C^{-1} \mathcal{E}(U\mu, U\mu)$. Reversing the role of $\mathcal{E}$ and $\bar{\mathcal{E}}$ gives the opposite inequality, and hence
$$ (C')^{-1} \mathcal{E}(U\mu, U\mu) \leq \bar{\mathcal{E}}(\bar{U}\mu, \bar{U}\mu) \leq C^{-1} \mathcal{E}(U\mu,U\mu).$$
Since this comparison holds when $\bar{\mathcal{E}}$ is the smallest closed extension of either the maximal and minimal cluster point of $(\mathcal{E}_N^\mathcal{I})_N$ in the Mosco topology, we find (\ref{eq:Greenforms}).
\end{proof}

\subsection{Comparison of discrete Dirichlet \& Green forms}

Recall that we are considering the free field on the outer Sierpinski carpet graph $\mathcal{G}_\infty$, while the convergence from discrete Dirichlet forms to the continuum one is based on the inner Sierpinski carpet graph
$\mathcal{I}_\infty$. To bridge this gap, we shall compare the discrete Dirichlet (and Green) forms on
$\mathcal{G}_\infty$ and on $\mathcal{I}_\infty$ in this subsection.

Observe (from Figure \ref{fig:SCGraph}) that for each ``center vertex'' $w \in I_\infty$, there is a unique set $\mathcal{C}(w)$ of $2^d$
``corner vertices'' in $V_\infty$ which are nearest neighbors of $w$, \emph{i.e.,} $\mathcal{C}(w) = \{x \in V_\infty: \|x-w\|=\sqrt{d}/2\}$. Let $\tilde{Q}: C(V_\infty;\mathbb{R}) \to
C(I_\infty;\mathbb{R})$ be the projection operator given by
$$ (\tilde{Q} f)(w) = \frac{1}{2^d} \sum_{x\in \mathcal{C}(w)} f(x).$$
As is customary, we introduce the discrete Dirichlet form on graph $\mathcal{G}_\infty$ by
$$
E_{\mathcal{G}_\infty}(f_1,f_2) = \frac{1}{2} \sum_{\substack{x,x'\in V_\infty\\x\sim x'}}
(f_1(x)-f_1(x'))(f_2(x)-f_2(x'))
$$
for all $f_1,f_2$ in the natural domain $\mathcal{D}(E_{\mathcal{G}_\infty})$. Let $G_{\mathcal{G}_N}: V_\infty\times V_\infty \to \mathbb{R}$ denote the Green's function killed upon exiting $\mathcal{G}_N$.

\begin{lemma}
For all $f\in \mathcal{D}(E_{\mathcal{G}_\infty})$,
\begin{equation}E_{\mathcal{G}_\infty}(f,f) \geq E_{\mathcal{I}_\infty}\left(\tilde{Q} f, \tilde{Q}
f\right).\label{eq:DDFCompare}\end{equation}
It follows that for all nonnegative functions $f$ on $V_N$,
\begin{equation}\sum_{x,x'\in V_N} G_{\mathcal{G}_N}(x,x')f(x)f(x') \leq 2^{4d} \sum_{w,w'\in I_N}
G_{\mathcal{I}_N}(w,w') (\tilde{Q} f)(w) (\tilde{Q} f)(w'). \label{eq:DGFCompare}\end{equation}
\label{lem:Ecomp}
\end{lemma}

\begin{proof}
Observe that for every $w_1,w_2 \in I_\infty$ with $w_1 \sim w_2$,
$$
(\tilde{Q} f)(w_1) - (\tilde{Q} f)(w_2) = \frac{1}{2^d} \sum_{\substack{x_1\in \mathcal{C}(w_1) \backslash \mathcal{C}(w_2) \\x_2\in
\mathcal{C}(w_2)\backslash \mathcal{C}(w_1)\\ z\in \mathcal{C}(w_1)\cap \mathcal{C}(w_2) \\ x_1\sim z \sim x_2}}
\left[\left(f(x_1)-f(z)\right)+\left(f(z)- f(x_2)\right)\right].
$$
The sum is over the difference of $f$ across $2^d$ edges in $\mathcal{G}_\infty$ which are parallel to the line segment $\overline{w_1 w_2}$, and whose vertices are nearest neighbors of either $w_1$ or $w_2$. Taking the square of both sides, and applying the inequality $(\sum_{k=1}^n
a_k)^2 \leq n(\sum_{k=1}^n a_k^2)$, we obtain
$$
\left[(\tilde{Q} f)(w_1) - (\tilde{Q} f)(w_2)\right]^2 \leq \frac{1}{2^d} \sum_{\substack{x_1\in \mathcal{C}(w_1)
\backslash \mathcal{C}(w_2) \\x_2\in \mathcal{C}(w_2)\backslash \mathcal{C}(w_1)\\ z\in \mathcal{C}(w_1)\cap \mathcal{C}(w_2) \\ x_1\sim z \sim x_2}}
\left[\left(f(x_1)-f(z)\right)^2+\left(f(z)-f(x_2)\right)^2\right].
$$
Upon summing over all $w_1,w_2 \in I_\infty$, we note that each edge in $\mathcal{G}_\infty$ contributes at most
$2^d$ terms to the RHS, that is,
\begin{eqnarray*}
E_{\mathcal{I}_\infty} (\tilde{Q}f, \tilde{Q}f) &=& \frac{1}{2} \sum_{w_1\sim w_2} \left[(\tilde{Q}f)(w_1) -
(\tilde{Q} f)(w_2)\right]^2 \\
&\leq& \frac{1}{2}\cdot \frac{1}{2^d} \cdot 2^d \sum_{\substack{x_1,x_2 \in V_\infty\\ x_1\sim x_2}}
(f(x_1)-f(x_2))^2 ~=~ E_{\mathcal{G}_\infty}(f,f).
\end{eqnarray*}
This proves (\ref{eq:DDFCompare}).

Next we turn to the Green form inequality (\ref{eq:DGFCompare}). Observe that for all nonnegative functions $f$,
$h$ on $V_N$,
\begin{eqnarray*}
\sum_{w\in I_N}(\tilde{Q} f)(w) (\tilde{Q} h)(w) &=& \frac{1}{2^{2d}}\sum_{w\in I_N} \sum_{x,x' \in \mathcal{C}(w)}
f(x)h(x') \\
&\geq& \frac{1}{2^{2d}}\sum_{w\in I_N} \sum_{x \in \mathcal{C}(w)} f(x)h(x) ~\geq~ \frac{1}{2^{2d}} \sum_{x\in V_N} f(x)h(x).
\end{eqnarray*}
By the reproducing property of Green's functions, we deduce that for all $h: V_\infty \to \mathbb{R}_+$ with
support in $V_N$,
$$
E_{\mathcal{I}_\infty}\left(\sum_{w\in I_N} G_{\mathcal{I}_N}(\cdot,w)(\tilde{Q} f)(w), \tilde{Q} h\right) \geq
\frac{1}{2^{2d}} E_{\mathcal{G}_\infty}\left(\sum_{x\in V_N} G_{\mathcal{G}_N}(\cdot,x)f(x),h\right).
$$
Taking $h=\sum_{x \in V_N} G_{\mathcal{G}_N}(\cdot,x)f(x)$ and again applying the reproducing property yields
$$E_{\mathcal{I}_\infty}\left(\sum_{w\in I_N} G_{\mathcal{I}_N}(\cdot, w)(\tilde{Q} f)(w), \sum_{x\in V_N}
\tilde{Q} G_{\mathcal{G}_N}(\cdot, x) f(x)\right) \geq \frac{1}{2^{2d}}\sum_{x,x'\in V_N}G_{\mathcal{G}_N}(x,x')
f(x) f(x').
$$
To simplify the notation, we introduce shorthands for the functions
$$
f_\alpha := \sum_{w\in I_N}G_{\mathcal{I}_N}(\cdot,w)(\tilde{Q} f)(w) \quad\text{and}\quad
f_\beta := \sum_{x\in V_N}\tilde{Q} G_{\mathcal{G}_N}(\cdot,x) f(x),
$$
on $I_N$. It is clear that 
$$E_{\mathcal{I}_\infty}(f_\alpha, f_\alpha) = \sum_{w,w' \in I_N} G_{\mathcal{I}_N}(w,w')(\tilde{Q} f)(w)
(\tilde{Q} f)(w').$$
Meanwhile, by the energy inequality (\ref{eq:DDFCompare}) we just proved,
$$E_{\mathcal{I}_\infty}(f_\beta, f_\beta) \leq E_{\mathcal{G}_\infty} \left(\sum_{x\in V_N}
G_{\mathcal{G}_N}(\cdot, x)f(x), \sum_{x'\in V_N} G_{\mathcal{G}_N}(\cdot,x')f(x')\right) = \sum_{x,x' \in V_N}
G_{\mathcal{G}_N}(x,x')f(x)f(x').$$
So putting everything together,
\begin{eqnarray*}
0 &\leq& E_{\mathcal{I}_\infty}\left(f_\alpha - \frac{1}{2^{2d}} f_\beta, ~f_\alpha -\frac{1}{2^{2d}} f_\beta\right) \\&=& E_{\mathcal{I}_\infty}(f_\alpha, f_\alpha) - \frac{2}{2^{2d}} E_{\mathcal{I}_\infty}\left(f_\alpha,
f_\beta\right) + \frac{1}{2^{4d}} E_{\mathcal{I}_\infty}(f_\beta,f_\beta)\\
&\leq& \sum_{w,w'\in I_N} G_{\mathcal{I}_N}(w,w')(\tilde{Q} f)(w) (\tilde{Q} f)(w') -
\frac{1}{2^{4d}}\sum_{x,x'\in V_N} G_{\mathcal{G}_N}(x,x')f(x)f(x'),
\end{eqnarray*}
which yields (\ref{eq:DGFCompare}).
\end{proof}

\subsection{The main lemma}

In this final subsection, we establish the limsup convergence of discrete Green forms on $\mathcal{G}_\infty$, which will play a crucial role in the main proofs.

Let $G_{\mathcal{G}_N}^\square: V_N \times V_N \to \mathbb{R}$ be the
restriction of $G_{\mathcal{G}_\infty}$ on $V_N \times V_N$; we have added a superscript $\square$ to
distinguish it from the Green's function on $\mathcal{G}_\infty$ killed upon exiting $\mathcal{G}_N$. Let us note that $G_{\mathcal{G}_N}^\square$ and $G_{\mathcal{G}_N}$, regarded as $N$-by-$N$ matrices, are both nonsingular and have matrix inverses. Also
introduce the probability measure $\eta_N:=\frac{1}{|V_N|}\mathbbm{1}_{V_N}$ on $V_N$. Define, for any $h\in \ell^1(V_N;\mathbb{R}_+) \cap \ell^\infty(V_N;\mathbb{R}_+)$,
$$U_N^{\mathcal{G}}\left(h \eta_N\right) := \rho_F^{-N} \frac{1}{|V_N|}\sum_{x\in V_N} G_{\mathcal{G}_N}(\cdot,x)
h(x),\quad U_N^{\mathcal{G} \square}\left(h \eta_N\right) := \rho_F^{-N} \frac{1}{|V_N|}\sum_{x\in V_N}
G_{\mathcal{G}_N}^\square(\cdot,x) h(x).$$
Writing $\mathcal{E}_N^{\mathcal{G}}= \rho_F^N E_{\mathcal{G}_\infty}$ for the
renormalized discrete Dirichlet form on $\mathcal{G}_\infty$, we have
$$
\mathcal{E}_N^\mathcal{G}\left(U_N^\mathcal{G} \left(h \eta_N\right) , U_N^\mathcal{G} \left(h \eta_N\right)
\right) =\rho_F^{-N} \frac{1}{|V_N|^2}\sum_{x,x'\in V_N} G_{\mathcal{G}_N}(x,x') h(x)h(x')
$$
by the reproducing property of $G_{\mathcal{G}_N}$. Meanwhile, let us abuse notation slightly and introduce the
quadratic form
$$
\mathcal{E}_N^\mathcal{G}(U_N^{\mathcal{G} \square} (h\eta_N), U_N^{\mathcal{G} \square} (h\eta_N)) :=\rho_F^{-N}
\frac{1}{|V_N|^2}\sum_{x,x'\in V_N} G_{\mathcal{G}_N}^\square(x,x') h(x)h(x'),
$$
as it is suggestive of another Green form. 

Recall from \S\ref{sec:KZ} that $(\mathcal{E}, \mathcal{F}_0)$ is a regular Dirichlet form on $L^2(F_\infty,\nu_\infty)$. Therefore $\mathcal{E}$ possesses a core $\mathcal{C}$ which is $\mathcal{E}_1$-dense in $\mathcal{F}_0$ and sup-norm-dense in $C_c(F_\infty)$. In particular, $\mathcal{C} =\mathcal{F}_0 \cap C_c(F_\infty)$ is a (special standard) core of $\mathcal{E}$ \cite{FOT}. 

\begin{lemma}[The main lemma]
For every $h \in L^1(F_\infty,\nu_\infty) \cap \mathcal{F}_0$ with $h\geq 0$, define $h_N: V_N \to \mathbb{R}$ by $h_N(\cdot) = h(\ell_F^{-N} \cdot)$. Then
the following hold:
\begin{enumerate}[label={(\roman*)}]
\item $\displaystyle \varlimsup_{N\to\infty} \mathcal{E}_N^\mathcal{G}(U_N^{\mathcal{G} \square} (h_N \eta_N)
,U_N^{\mathcal{G} \square} (h_N \eta_N)) = \varlimsup_{N\to\infty} \mathcal{E}_N^\mathcal{G}(U_N^\mathcal{G} (h_N
\eta_N) ,U_N^\mathcal{G} (h_N \eta_N))$.
\end{enumerate}
For some $(\mathcal{E},\mathcal{F}) \in \mathfrak{E}$:
\begin{enumerate}[label={(\roman*)}]
\setcounter{enumi}{1} 
\item There exists a constant $\cc{2.6}$ such that 
$$
\varlimsup_{N\to\infty} \mathcal{E}_N^\mathcal{G}\left(U_N^{\mathcal{G} \square}
\left(\left(\frac{d\mu}{d\nu}\right)_N \eta_N\right),U_N^{\mathcal{G} \square} \left(\left(\frac{d\mu}{d\nu}\right)_N
\eta_N\right)\right) \leq \cc{2.6} \mathcal{E}(U\mu, U\mu)$$
for all $\mu \in \mathcal{M}_{0,{\rm ac}}^{(0)}(F)$.
\item There exists a constant $\cc{2.7}$ such that
$$ \varlimsup_{N\to\infty}\rho_F^N\left\langle \mathbbm{1}_{V_N}, \sum_{x\in
V_N}(G_{\mathcal{G}_N}^\square)^{-1}(\cdot,x) \mathbbm{1}_{V_N}(x) \right\rangle_{V_N} \leq \cc{2.7} {\rm
Cap}_{\mathcal{E}}(F),$$ 
where $\left(G_{\mathcal{G}_N}^\square\right)^{-1}$ denotes the matrix inverse of
$G_{\mathcal{G}_N}^\square$, and ${\rm Cap}_\mathcal{E}(F)$ denotes the $0$-capacity of $F$ with respect to $\mathcal{E}$.
\end{enumerate}
\label{lem:VNConv}
\end{lemma}

For some general facts about the $0$-order capacity (which will be used in the proof of Part (iii)), please see the Appendix.

\begin{proof}
(i): We need two facts. The first is the observation that $\displaystyle \uparrow
\lim_{N\to\infty}G_{\mathcal{G}_N}(x,x') = G_{\mathcal{G}_\infty}(x,x')$ for all $x,x' \in V_\infty$, since the
first exit time from $\mathcal{G}_N$ increases unboundedly with $N$. The second is the following two-sided Green's function estimate on
$\mathcal{G}_\infty$, proved in \cite[Theorem 5.3]{BBSCGraph}: there exist constants $\cc{5}$, $\cc{6}$ such that for
all $x,x' \in V_\infty$ with $d_{\mathcal{G}_\infty}(x,x')\geq 1$,
$$
\cc{5} \cdot d_{\mathcal{G}_\infty}(x,x')^{d_w-d_h} \leq G_{\mathcal{G}_\infty}(x,x') \leq \cc{6}\cdot  d_{\mathcal{G}_\infty}(x,x')^{d_w-d_h}.
$$
Note that $d_w-d_h<0$ as $\mathcal{G}_\infty$ is a transient Sierpinski carpet graph. Therefore
$$
\epsilon_N(x,x'):=\frac{G_{\mathcal{G}_\infty}(x,x') - G_{\mathcal{G}_N}(x,x')}{d_{\mathcal{G}_\infty}(x,x')^{d_w-d_h}}
$$
is bounded above by $\cc{6}$ (and bounded below by $0$), and $\displaystyle \lim_{N\to\infty} \epsilon_N(x,x')=0$ pointwise. 

Let us take $h\geq 0$, $h\in L^1(F_\infty,\nu_\infty) \cap \mathcal{C}$, where $\mathcal{C}=\mathcal{F}_0 \cap C_c(F_\infty)$ is a core of the Dirichlet form $(\mathcal{E},\mathcal{F}_0)$. When estimating the difference between the two sides of the equation in Part (i), we get
\begin{eqnarray}
\nonumber 0 &\leq& \varlimsup_{N\to\infty}\left[\mathcal{E}_N^\mathcal{G}(U_N^{\mathcal{G} \square} (h_N \eta_N),
U_N^{\mathcal{G} \square} (h_N \eta_N)) - \mathcal{E}_N^\mathcal{G}(U_N^\mathcal{G} (h_N \eta_N), U_N^\mathcal{G} (h_N
\eta_N))\right]\\
\nonumber &=&\varlimsup_{N\to\infty}\rho_F^{-N} \frac{1}{|V_N|^2}\sum_{x,x'\in V_N}
\left[G_{\mathcal{G}_\infty}(x,x')-G_{\mathcal{G}_N}(x,x')\right] h_N(x)h_N(x') \\
\nonumber &=& \varlimsup_{N\to\infty}\rho_F^{-N} \frac{1}{|V_N|^2} \sum_{x,x'\in
V_N} \epsilon_N(x,x') \cdot d_{\mathcal{G}_\infty}(x,x')^{d_w-d_h}h_N(x)h_N(x')\\
\nonumber &\leq& C \varlimsup_{N\to\infty}\rho_F^{-N} \frac{1}{|V_N|^2} \sum_{y,y' \in
\ell_F^{-N} V_N} \rho_F^N \epsilon_N(\ell_F^N y, \ell_F^N y') \|y-y'\|^{d_w-d_h} h(y)h(y')\\
\label{eq:intermediate} &=& C\varlimsup_{N\to\infty} \int_{F\times F} \epsilon_N(\ell_F^N y, \ell_F^N y') \frac{h(y)h(y') d\mathfrak{m}_N(y) d\mathfrak{m}_N(y)}{\|y-y'\|^{d_h-d_w}},
\end{eqnarray}
where $\displaystyle \mathfrak{m}_N =\frac{1}{|V_N|}\mathbbm{1}_{\ell_F^{-N} V_N}$ is a probability measure on $F$, and $\mathfrak{m}_N$ converges weakly to $\nu$. Now
$$
\int_{F\times F} \frac{h(y)h(y') d\nu(y) d\nu(y')}{\|y-y'\|^{d_h-d_w}} \leq \|h\|^2_\infty \int_{F\times F}\frac{d\nu(y)d\nu(y')}{\|y-y'\|^{d_h-d_w}} < \infty,
$$
where we use a fact from geometric measure theory (see \emph{e.g.} \cite[Ch. 8]{Mattila}) that since $\nu$ is a $d_h$-dimensional Hausdorff measure with respect to the
Euclidean norm $\|\cdot\|$ on the compact metric space $(F,\|\cdot\|)$,
$$
\int_{F\times F} \frac{d\nu(y) d\nu(y')}{\|y-y'\|^\alpha}<\infty
$$
for any $\alpha<d_h$. We also have $\|h\|_\infty <\infty$ since $h \in C_c(F_\infty)$. So by the reverse Fatou's lemma for weakly converging measures (\emph{cf.} \cite{Serfozo,Schal}; see also \cite[Theorem 1.1]{FatouWeak} for the statement and proof), the RHS of (\ref{eq:intermediate}) is bounded above by
$$
C \int_{F\times F} \varlimsup_{\substack{N\to\infty\\ y_1, y_2 \in \ell_F^{-N} V_N \\y_1\to y, y_2\to y'}} \left(\frac{\epsilon_N(\ell_F^N y_1, \ell_F^N y_2)h(y_1) h(y_2)}{\|y_1-y_2\|^{d_h-d_w}} \right)d\nu(y) d\nu(y')=0.
$$
By a density argument this estimate extends to all $h\geq 0$ with $h \in L^1(F_\infty,\nu_\infty) \cap \mathcal{F}_0$. This proves (i).

(ii): By Part (i) and then (\ref{eq:DGFCompare}), there exists a constant $C(d)$ such that
\begin{eqnarray*}
\varlimsup_{N\to\infty}\mathcal{E}_N^\mathcal{G}\left(U_N^{\mathcal{G} \square}\left(h_N
\eta_N\right),U_N^{\mathcal{G} \square}\left(h_N \eta_N\right)\right) &=& \varlimsup_{N\to\infty}
\mathcal{E}_N^\mathcal{G}\left(U_N^{\mathcal{G}}\left(h_N \eta_N\right),U_N^{\mathcal{G}}\left(h_N
\eta_N\right)\right)\\
&\leq& C \varlimsup_{N\to\infty}\mathcal{E}_N^\mathcal{I}\left(U_N^\mathcal{I} \left(\tilde{Q} h_N
\nu_N\right), U_N^\mathcal{I}\left(\tilde{Q} h_N \nu_N\right)\right).
\end{eqnarray*}
for all $h\geq 0$ with $h\in L^1(F_\infty,\nu_\infty)\cap \mathcal{F}_0$. We claim that $\tilde Q h_N$ can be replaced by $\tilde P_N h$ in the above inequality. Indeed, if $h \in L^1(F_\infty,\nu_\infty) \cap \mathcal{F}_0 \cap C_c(F_\infty)$, then by continuity
\begin{eqnarray*}
\lim_{N\to\infty}\|\tilde Q h_N - \tilde P_N h\|_{L^1(I_N,\nu_N)} &\leq&\lim_{N\to\infty} \frac{1}{m_F^N}\sum_{w\in I_N}\frac{1}{\left|\nu(\Psi^{(N)}_w \cap F)\right|} \int_{\Psi^{(N)}_w \cap F} \left|h(y) - \frac{1}{2^d}\sum_{x\in \mathcal{C}(w)} h\left(\frac{x}{\ell_F^N}\right)\right| d\nu(y) \\
&\leq& \lim_{N\to\infty}\sup_{w\in I_N} \sup_{y \in \Psi^{(N)}_w \cap F} \left|h(y) - \frac{1}{2^d}\sum_{x\in \mathcal{C}(w)} h\left(\frac{x}{\ell_F^N}\right)\right|~=~0.
\end{eqnarray*}
Since $\mathcal{F}_0 \cap C_c(F_\infty)$ is a core of $\mathcal{E}$, this result extends to all $h\in L^1(F_\infty,\nu_\infty) \cap \mathcal{F}_0$ by a density argument. Hence
\begin{eqnarray*}
&&  \varlimsup_{N\to\infty}\mathcal{E}_N^\mathcal{I}\left(U_N^\mathcal{I} \left(\tilde{Q} h_N
\nu_N\right), U_N^\mathcal{I}\left(\tilde{Q} h_N \nu_N\right)\right) ~=~ \varlimsup_{N\to\infty}\left\langle U_N^\mathcal{I} \left(\tilde{Q} h_N
\nu_N\right), \tilde{Q} h_N \nu_N\right\rangle_{\mathcal{I}_N} \\ &\leq& \varlimsup_{N\to\infty} \left\langle U_N^\mathcal{I} \left((\tilde{Q} h_N
+\tilde P_N h) \nu_N\right), (\tilde{Q} h_N - \tilde P_N h)\nu_N \right\rangle_{\mathcal{I}_N} + \varlimsup_{N\to\infty} \left\langle U_N^\mathcal{I} \left(\tilde{P}_N h
\nu_N\right), \tilde{P}_N h \nu_N\right\rangle_{\mathcal{I}_N}\\
&=& \varlimsup_{N\to\infty} \left\langle U_N^\mathcal{I} \left(\tilde{P}_N h
\nu_N\right), \tilde{P}_N h \nu_N\right\rangle_{\mathcal{I}_N} ~=~  \varlimsup_{N\to\infty}\mathcal{E}_N^\mathcal{I}\left(U_N^\mathcal{I} \left(\tilde P_N h
\nu_N\right), U_N^\mathcal{I}\left(\tilde P_N h \nu_N\right)\right).
\end{eqnarray*}

Now put $h=\frac{d\mu}{d\nu}$ for some $\mu\in \mathcal{M}_{0,{\rm ac}}^{(0)}(F_\infty)$. It follows from the preceding discussions and Theorem \ref{thm:Greenform} that
$$
\varlimsup_{N\to\infty} \mathcal{E}_N^\mathcal{G}\left(U_N^{\mathcal{G} \square}
\left(\left(\frac{d\mu}{d\nu}\right)_N \eta_N\right),U_N^{\mathcal{G} \square} \left(\left(\frac{d\mu}{d\nu}\right)_N
\eta_N\right)\right) \leq C \varlimsup_{N\to\infty} \mathcal{E}_N^{\mathcal{I}}\left(U_N^{\mathcal{I}}(\tilde P_N \mu),
U_N^{\mathcal{I}}(\tilde P_N \mu)\right) \leq C' \mathcal{E}(U\mu,U\mu),
$$
where $C'=C \cc{2.5}$.

(iii): We recognize that for all $h: V_N\to\mathbb{R}_+$,
\begin{equation}
\left\langle \rho_F^{-N}\sum_{x\in V_N}G_{\mathcal{G}_N}^\square(\cdot,x)(h\eta_N)(x), h\eta_N \right \rangle_{V_N} \geq \left\langle \rho_F^{-N}\sum_{x\in V_N}G_{\mathcal{G}_N}(\cdot,x)(h\eta_N)(x), h\eta_N \right \rangle_{V_N}.
\label{eq:GreenIneq}
\end{equation}
Fixing $f : V_N \to\mathbb{R}$, we let $\displaystyle \eta = \rho_F^N\sum_{x\in V_N} (G_{\mathcal{G}_N})^{-1}(\cdot,x)f(x)$ and $\displaystyle \eta^\square = \rho_F^N\sum_{x\in V_N} (G_{\mathcal{G}_N}^\square)^{-1}(\cdot,x)f(x)$. (Note that the matrix inverses are well-defined.) Then upon applying the Cauchy-Schwarz inequality and (\ref{eq:GreenIneq}), we find
\begin{eqnarray*}
\left\langle f, \rho_F^N \sum_{x\in V_N}(G_{\mathcal{G}_N}^\square)^{-1}(\cdot,x) f(x)\right\rangle_{V_N}^2 &=& \left\langle \rho_F^{-N}\sum_{x\in V_N}G_{\mathcal{G}_N}(\cdot,x)\eta(x), \eta^o\right\rangle_{V_N}^2 \\
&\leq& \left\langle \rho_F^{-N} \sum_{x\in V_N} G_{\mathcal{G}_N}(\cdot,x)\eta(x), \eta\right\rangle_{V_N} \left\langle \rho_F^{-N}\sum_{x\in V_N}G_{\mathcal{G}_N}(\cdot,x)\eta^o(x), \eta^o\right\rangle_{V_N}\\
&\leq& \left\langle \rho_F^{-N} \sum_{x\in V_N} G_{\mathcal{G}_N}(\cdot,x)\eta(x), \eta\right\rangle_{V_N} \left\langle \rho_F^{-N}\sum_{x\in V_N}G_{\mathcal{G}_N}^\square(\cdot,x)\eta^o(x), \eta^o\right\rangle_{V_N}\\
&=& \left\langle f, \rho_F^N \sum_{x\in V_N}(G_{\mathcal{G}_N})^{-1}(\cdot,x)f(x)\right\rangle_{V_N} \left\langle f, \rho_F^N \sum_{x\in V_N}(G_{\mathcal{G}_N}^\square)^{-1}(\cdot,x)f(x)\right\rangle_{V_N}.
\end{eqnarray*}
Hence for all $f: V_N \to \mathbb{R}$,
$$
\left\langle f,\rho_F^N\sum_{x\in V_N}(G_{\mathcal{G}_N}^\square)^{-1}(\cdot,x)f(x)\right\rangle_{V_N} \leq \left\langle f,\rho_F^N\sum_{x\in V_N}(G_{\mathcal{G}_N})^{-1}(\cdot,x)f(x)\right\rangle_{V_N}.
$$
In particular,
\begin{eqnarray}
\label{eq:ineq1} \qquad \varlimsup_{N\to\infty}\rho_F^N \left\langle \mathbbm{1}_{V_N}, \sum_{x\in V_N}(G_{\mathcal{G}_N}^\square)^{-1}(\cdot,x) \mathbbm{1}_{V_N}(x)\right\rangle_{V_N} &\leq& \varlimsup_{N\to\infty}\rho_F^N \left\langle \mathbbm{1}_{V_N}, \sum_{x\in V_N}(G_{\mathcal{G}_N})^{-1}(\cdot,x) \mathbbm{1}_{V_N}(x)\right\rangle_{V_N}\\
\nonumber &=& \varlimsup_{N\to\infty}\langle \mathbbm{1}_{V_N}, \mu_{V_N}\rangle_{V_N},
\end{eqnarray}
where $\mu_{V_N}$ is the equilibrium measure on $V_N$ with respect to $\mathcal{E}_N^\mathcal{G}$. (The equilibrium measure and the equilibrium potential of a set $B$ with respect to the Dirichlet form $\mathcal{E}$ is defined in Proposition \ref{prop:capacity} in the Appendix.) The equality in the second line is just a direct calculation by
$$  U_N^\mathcal{G} \mu_{V_N} = \rho_F^{-N} \frac{1}{|V_N|} \sum_{x\in V_N} G_{\mathcal{G}_N}(\cdot,x) \frac{d\mu_{V_N}}{d\eta_N}(x) =1 \quad \text{on}~V_N,$$
\emph{cf.} Proposition \ref{prop:capacity}(ii). From this it also follows that $\rho_F^{-N} \mu_{V_N}(V_N) \leq \underline{G}^{-1}$, and
\begin{eqnarray*}
U_N^\mathcal{I} \tilde Q \mu_{V_N} &=& \rho_F^{-N} \frac{1}{m_F^N} \sum_{w\in I_N} G_{\mathcal{I}_N}(\cdot, w) \frac{1}{2^d}\sum_{x\in \mathcal{C}(w)}\left(\frac{d\mu_{V_N}}{d\eta_N}\right)(x)\\
&\leq& \frac{|V_N|}{m_F^N} \|G_{\mathcal{I}_\infty}\|_\infty \rho_F^{-N} \mu_{V_N}(V_N) ~\leq ~ \frac{|V_N|}{m_F^N} \frac{\|G_{\mathcal{I}_\infty}\|_\infty}{\underline{G}} \quad \text{on}~I_N.
\end{eqnarray*}
So if we let $\hat\mu_{V_N} := \left( \frac{|V_N|}{m_F^N} \frac{\|G_{\mathcal{I}_\infty}\|_\infty}{\underline{G}} \right)^{-1}\mu_{V_N}$, then $U_N^\mathcal{I} \tilde Q \hat\mu_{V_N} \leq 1$ on $I_N$. Hence by Proposition \ref{prop:capacity}(iv),
\begin{eqnarray*}
\langle \mathbbm{1}_{V_N}, \mu_{V_N}\rangle_{V_N} &\leq& 2^d \frac{m_F^N}{|V_N|}\langle \mathbbm{1}_{I_N}, \tilde Q \mu_{V_N} \rangle_{I_N} 
~=~ 2^d \frac{\|G_{\mathcal{I}_\infty}\|_\infty}{\underline{G}} \langle \mathbbm{1}_{I_N}, \tilde{Q}\hat\mu_{V_N}\rangle_{I_N} \\
&\leq& 2^d \frac{\|G_{\mathcal{I}_\infty}\|_\infty}{\underline{G}} \langle \mathbbm{1}_{I_N}, \mu_{I_N}\rangle_{I_N}  ~=~ 2^d \frac{\|G_{\mathcal{I}_\infty}\|_\infty}{\underline{G}} \mathcal{E}_N^\mathcal{I}(e_{I_N}, e_{I_N}),
\end{eqnarray*}
where $\mu_{I_N}$ and $e_{I_N}$ are, respectively, the equilibrium measure and equilibrium potential of $I_N$ with respect to $\mathcal{E}_N^\mathcal{I}$. Putting $\cc{2.8}:=2^d \frac{\|G_{\mathcal{I}_\infty}\|_\infty}{\underline{G}}$, and applying Proposition \ref{prop:capacity}(i) and Proposition \ref{prop:KZ}(ii), we find
\begin{equation}
\varlimsup_{N\to\infty}\langle \mathbbm{1}_{V_N},\mu_{V_N}\rangle_{V_N} \leq \cc{2.8}\varlimsup_{N\to\infty} \mathcal{E}_N^\mathcal{I}\left(e_{I_N},e_{I_N}\right)\leq \cc{2.8}\varlimsup_{N\to\infty} \mathcal{E}_N^\mathcal{I}\left(\tilde P_N e_F, \tilde P_N e_F\right) \leq \cc{2.8}\cc{2.2}^{-1} \mathcal{E}(e_F, e_F).
\label{eq:ineq2}
\end{equation}
Inequalities (\ref{eq:ineq1}) and (\ref{eq:ineq2}) together imply the result.
\end{proof}

\section{Proof of Theorem \ref{thm:LD}} \label{sec:LD}

\emph{Notation.} In the next two sections, $\Phi: \mathbb{R} \to [0,1]$, defined by
$$\Phi(a) = \frac{1}{\sqrt{2\pi}} \int_{-\infty}^a e^{-\xi^2/2} d\xi,$$
stands for the cdf of a standard normal random variable. For any measurable subset $S$ of $V_\infty$, we denote
by $\mathscr{F}_S:=\sigma\{\varphi_x: x\in S\}$ the sigma-algebra generated by the free field on $S$, and by
$\Omega^+_S := \{ \varphi_x\geq 0 ~\text{for all}~x\in S\}$ the event that the field is nonnegative everywhere on
$S$. Finally, we fix an element $(\mathcal{E},\mathcal{F})$ from the family $\mathfrak{E}$ of local, regular, conservative, non-zero Dirichlet forms on $L^2(F_\infty,\nu_\infty)$ which are invariant under the local symmetries of the carpet.

\subsection{Lower bound} \label{subsec:LDLB}

Let $\alpha > d_s(F) \cdot \overline{G}$, where $d_s(F) = 2\frac{\log m_F}{\log t_F}$ is the spectral dimension of the generalized Sierpinski carpet $F$ (note that $d_s(F)>2$), and $\overline{G} := \sup_{x\in V_\infty} G_{\mathcal{G}_\infty}(x,x)$. Denote by
$\mathbb{P}_N$ the law of the free field on $\mathcal{G}_\infty$ with mean $\sqrt{\alpha \log t_F}\sqrt{N}$ and
covariance $G_{\mathcal{G}_\infty}$.

First we wish to show that $\lim_{N\to\infty}\mathbb{P}_N(\Omega_{V_N}^+) = 1$. Observe that for any $x \in
V_\infty$,
$$
\mathbb{P}_N(\varphi_x<0) = \mathbb{P}(\varphi_x < -\sqrt{\alpha N\log t_F}) = \Phi\left(-\sqrt{\frac{\alpha N \log
t_F}{G_{\mathcal{G}_\infty}(x,x)}}\right),
$$
Using the fact that $G_{\mathcal{G}_\infty}(x,x) \leq \overline{G}$ and $\Phi(a) \leq \frac{1}{2} e^{-a^2/2}$ for
$a\leq 0$, we deduce that
$$
\mathbb{P}_N(\varphi_x<0) \leq \frac{1}{2} t_F^{-(N\alpha)/(2 \overline{G})}.
$$
It follows that
$$\mathbb{P}_N\left((\Omega_{V_N}^+)^c\right) = \mathbb{P}_N\left(\bigcup_{x\in V_N}\{\varphi_x<0\}\right) \leq
\sum_{x\in V_N}\mathbb{P}_N(\varphi_x<0) \leq c \left(m_F t_F^{-\frac{\alpha}{2 \overline{G}}}\right)^N
\underset{N\to\infty}{\longrightarrow} 0,$$
which is what we want.

Next we adopt the relative entropy argument as used in the proof of \cite[Lemma 2.3]{BDZ95}. Let $\displaystyle \Pi_N =
\left. \frac{d\mathbb{P}_N}{d\mathbb{P}}\right|_{\mathscr{F}_{V_N}}$. Introduce the relative entropy of
$\mathbb{P}_N$ to $\mathbb{P}$ restricted to $V_N$ by
$${\rm Ent}_{V_N}(\mathbb{P}_N|\mathbb{P})= \int_{\mathbb{R}^{V_\infty}} \Pi_N \log(\Pi_N) d\mathbb{P} =
\frac{1}{2}\alpha N \log t_F \left\langle \mathbbm{1}_{V_N}, \sum_{x\in V_N}(G_{\mathcal{G}_N}^\square)^{-1}
(\cdot,x)\mathbbm{1}_{V_N}(x)\right\rangle_{V_N}, $$
where $(G_{\mathcal{G}_N}^\square)^{-1}$ denotes the matrix inverse of $(G_{\mathcal{G}_N}^\square) := \left.
G_{\mathcal{G}_\infty}\right|_{V_N \times V_N}$. Applying the entropy inequality 
$$
\log \left(\frac{\mathbb{P}(\Omega_{V_N}^+)}{\mathbb{P}_N(\Omega_{V_N}^+)} \right) \geq  -\frac{1}{\mathbb{P}_N(\Omega_{V_N}^+)} \left({\rm Ent}_{V_N}(\mathbb{P}_N|\mathbb{P})+e^{-1}\right),
$$
\emph{cf.} the end of the proof of \cite[Lemma 2.3]{BDZ95}, we obtain
\begin{eqnarray}
\nonumber &&\varliminf_{N\to\infty} \frac{\log\mathbb{P}(\Omega_{V_N}^+)}{\rho_F^{-N} N \log t_F} \\
\nonumber &\geq& \varliminf_{N\to\infty} \left[-\frac{1}{\mathbb{P}_N(\Omega_{V_N}^+)} \left(\frac{{\rm
Ent}_{V_N}(\mathbb{P}_N|\mathbb{P})+e^{-1}}{\rho_F^{-N} N \log t_F}\right) + \frac{\log
\mathbb{P}_N(\Omega_{V_N}^+)}{\rho_F^{-N} N \log t_F} \right]\\
\nonumber &\geq& \varliminf_{N\to\infty} \left[-\frac{1}{\mathbb{P}_N(\Omega_{V_N}^+)}\frac{{\rm
Ent}_{V_N}(\mathbb{P}_N|\mathbb{P})}{\rho_F^{-N} N \log t_F} \right] + \varliminf_{N\to\infty}
\left[-\frac{1}{\mathbb{P}_N(\Omega_{V_N}^+)}\frac{e^{-1}}{\rho_F^{-N} N \log t_F} \right] +
\varliminf_{N\to\infty} \frac{\log \mathbb{P}_N(\Omega_{V_N}^+)}{\rho_F^{-N} N \log t_F}\\
 &\geq& -\left(\varlimsup_{N\to\infty} \frac{1}{\mathbb{P}_N(\Omega_{V_N}^+)}\cdot
\varlimsup_{N\to\infty}\frac{{\rm Ent}_{V_N}(\mathbb{P}_N|\mathbb{P})}{\rho_F^{-N} N \log t_F}\right) +0 +0 
~\geq~ -\frac{1}{2}\alpha \cc{2.7} {\rm
Cap}_{\mathcal{E}}(F)\label{ineq:LDLB}
\end{eqnarray}
by Lemma \ref{lem:VNConv}(iii). By making $\alpha$ arbitrarily close to $d_s(F)\cdot\overline{G}$, we
obtain the desired lower bound.

\begin{remark}
If we instead use a constant multiple of the original Dirichlet form $(\gamma\mathcal{E},\mathcal{F})$,
$\gamma>0$, inequality (\ref{ineq:LDLB}) will hold under the substitutions $\cc{2.7} \to \gamma^{-1} \cc{2.7}$ and ${\rm
Cap}_{\mathcal{E}}(F) \to {\rm Cap}_{\gamma\mathcal{E}}(F)$.
\label{rem:LDLB}
\end{remark}

\subsection{Upper bound} \label{subsec:LDUB}

Just as in the $\mathbb{Z}^d$ setting \cite{BDZ95}, the proof of the upper bound involves a series of \emph{coarse graining} and \emph{conditioning} arguments on the free field $\{\varphi_x\}_{x\in V_\infty}$, though some modifications are needed to account for the fractal geometry.

\emph{Notation.} If $\mathcal{G}=(V(\mathcal{G}), \sim)$ is a finite subgraph of a larger graph $\mathcal{G}_0 = (V(\mathcal{G}_0),\sim)$, then we denote the set of \emph{peripheral vertices} of $\mathcal{G}$ by
$$\partial \mathcal{G} := \{x \in V(\mathcal{G}): x\sim y ~\text{for some}~ y\in V(\mathcal{G}_0) \backslash V(\mathcal{G})\}.$$
The \emph{interior} of the graph $\mathcal{G}$ will thusly be defined by $\mathring{\mathcal{G}} := (V(\mathcal{G}) \backslash \partial \mathcal{G}, \sim)$.

Following \S\ref{sec:fractal}, we denote by $\mathcal{Q}_j(F_j)$ ($j\in \mathbb{N}$) the collection of closed cubes of side $\ell_F^{-j}$ whose vertices are in $\ell_F^{-j} \mathbb{Z}^d$, and which are contained in $F_j$. Then to each $\bar{Q}\in \mathcal{Q}_j(F_j)$ corresponds a unique vector ${\bf p}=(p_1,\cdots, p_d) \in (\mathbb{N}_0)^d$ such that $\bar{Q} = \left[p_1 \ell_F^{-j}, (p_1+1)\ell_F^{-j}\right] \times \cdots \times \left[p_d \ell_F^{-j}, (p_d+1) \ell_F^{-j}\right]$. Keeping with this notation, we define two related cubes derived from $\bar{Q}$: 
\begin{eqnarray*}
Q_\lefthalfcup &=& \left[p_1 \ell_F^{-j}, (p_1+1)\ell_F^{-j}\right) \times \cdots \times \left[p_d \ell_F^{-j}, (p_d+1)\ell_F^{-j}\right),\\
Q &=& Q_\lefthalfcup \cup \left( \bar{Q} ~\backslash~ \bigcup_{\substack{\bar{Q}'\in \mathcal{Q}_j(F_j) \\ \bar{Q}' \neq \bar{Q}}} Q'_\lefthalfcup\right). 
\end{eqnarray*}
The point of introducing the cube $Q$ is to ensure that the adjoining face between cubes is assigned to only one cube in a consistent manner. Let $\mathcal{Q}_j^\circ(F_j)$ be the totality of all $Q$. Observe that $F_j = \bigcup_{Q\in \mathcal{Q}_j^\circ(F_j)} Q$, and that $Q_1 \cap Q_2 = \emptyset$ for any $Q_1, Q_2 \in \mathcal{Q}_j^\circ(F_j)$ with $Q_1\neq Q_2$. 

Next we introduce, for each $k\leq N$, the following collections of $k$th-level subgraphs of $\mathcal{G}_N$:
\begin{eqnarray*}
S_k(\mathcal{G}_N) &:=&\left\{\ell_F^N \bar{Q} \cap \mathcal{G}_N: \bar{Q} \in \mathcal{Q}_{N-k}(F_{N-k})\right\},\\
S_k^\circ(\mathcal{G}_N) &:=&\left\{\ell_F^N Q \cap \mathcal{G}_N: Q \in \mathcal{Q}_{N-k}^\circ(F_{N-k})\right\}.
\end{eqnarray*}
By construction, there is a bijection $\iota_k : S_k^\circ(\mathcal{G}_N) \to \mathcal{Q}_{N-k}^\circ(F_{N-k})$ which maps each subgraph $\mathfrak{g} \in S_k^\circ(\mathcal{G}_N)$ to a subcube $Q \in \mathcal{Q}_{N-k}^\circ(F_{N-k})$. We will refer to $\iota_k$ often.

Now let us fix a sufficiently large $k \in \mathbb{N}$, and designate a vertex $x_0 \in V_k \backslash \partial \mathcal{G}_k$ as the ``representative interior point'' of $\mathcal{G}_k$. We do not insist on
where $x_0$ is located within $V_k$, so long as it stays away from the periphery $\partial \mathcal{G}_k$. (Contrast this
setup with previous works on $\mathbb{Z}^d$ \cite{BDZ95,Kurt}, where it is natural to designate the center vertex of each block cell as the
representative interior point.). Then for any $N>k$, let
\begin{eqnarray*}
\allrip{N} &=& \{x\in V_N: x=x_0 + \ell_F^k {\bf p}~\text{for some}~{\bf p}\in (\mathbb{N}_0)^d \} \quad\text{and}\\
\mathcal{D}_N &=& \{ x \in V_N: \exists i\in \{1,\cdots,d\} ~\text{such that} ~x_i =p \ell_F^k ~\text{for some}~ p\in \mathbb{N}_0\}
\end{eqnarray*}
be, respectively, the set of all representative interior points and $k$th-level boundary points in $V_N$; see Figure \ref{fig:conditioning}. Note that $|\allrip{N}|= m_F^{N-k}$.

\begin{figure}
\centering
\includegraphics{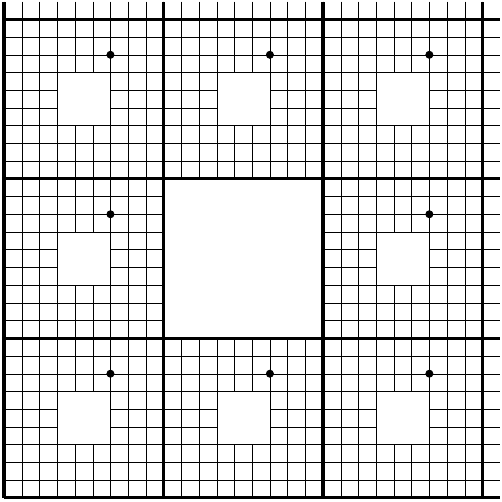}
\caption{The coarse-graining and conditioning scheme on the outer Sierpinski carpet graph $\mathcal{G}_\infty$. Vertices indicated by filled dots are the representative interior points ($\mathcal{C}_N$), while vertices covered by the solid lines (the conditioning grid) are where the free field $\varphi$ is conditioned upon ($\mathcal{D}_N$).}
\label{fig:conditioning}
\end{figure}

With the setup complete, we can proceed with the main arguments. \emph{Coarse graining} means that we are sampling the free field $\varphi$ at only one vertex from each subgraph $\mathfrak{g} \in S_k(\mathcal{G}_N)$. On top of that, we will analyze these Gaussian random variables conditional upon the sigma-algebra 
$\mathscr{F}_{\mathcal{D}_N}$ generated by the free field on the ``conditioning grid'' $\mathcal{D}_N$. 

A key property we will use is the random walk representation of the free field: if $S$ is a measurable subset of $V_N$, then the law of $\varphi_x$ conditional upon $\mathscr{F}_S$ is  Gaussian with
\begin{equation}\label{eq:RWrep}
\mathbb{E}[\varphi_x | \{\varphi_y = g_y\}_{y\in S}] = \sum_{y\in S} \mathbb{P}^x[X_{\tau_S}=y] g_y, \quad {\rm Var}[\varphi_x| \{\varphi_y=g_y\}_{y\in S}] = R_{\rm eff}(x,S),
\end{equation}
where $(X_t)_{t\geq 0}$ is the continuous-time random walk on the graph, $\tau_S$ denotes the first hitting time of $S$, and $R_{\rm eff}(A,B)$ is the effective resistance between two measurable subsets $A$ and $B$ of the graph.

From property (\ref{eq:RWrep}), we deduce that under $\mathbb{P}(\cdot|\mathscr{F}_{\mathcal{D}_N})$,
$\{\varphi_x : x\in \allrip{N}\}$ are independent Gaussian random variables with
mean $\mathbb{E}(\varphi_x | \mathscr{F}_{\mathcal{D}_N})=:\mu_x$ and identical variance $R_{\rm eff}(x_0, \partial \mathcal{G}_k)=G_{\mathring{\mathcal{G}}_k}(x_0, x_0)$, where $G_{\mathring{\mathcal{G}_k}}$ denotes the Green's function killed upon exiting $\mathring{\mathcal{G}}_k$. In addition we will use the following two properties of $\mu_x$ later:
\begin{itemize}
\item Under $\mathbb{P}$, the distribution of $(\varphi_y)_{y\in \mathcal{D}_N}$ is jointly Gaussian, and (\ref{eq:RWrep}) implies that $\mu_x$ is a sum of jointly Gaussian random variables, which is another Gaussian.
\item On the event $\Omega_{\mathcal{D}_N}^+$, (\ref{eq:RWrep}) implies that $\mu_x$ is the sum of nonnegative numbers, and hence a nonnegative number.
\end{itemize}

%

Let us now carry out the estimate of $\mathbb{P}(\Omega^+_{V_N})$. First of all,
\begin{equation}
\mathbb{P}(\Omega_{V_N}^+) \leq \mathbb{P}\left(\Omega_{\allrip{N}}^+ \cap\Omega_{\mathcal{D}_N}^+\right) =
\mathbb{E}\left(\prod_{x\in \allrip{N}} \mathbb{P}[\varphi_x\geq 0|\mathscr{F}_{\mathcal{D}_N}]\cdot
\mathbbm{1}_{\Omega_{\mathcal{D}_N}^+} \right),
\label{eq:P+Uppbound}
\end{equation}
where the equality comes from a basic identity for conditioned random variables and the independence of
$\{\varphi_x: x\in \allrip{N}\}$ under $\mathbb{P}(\cdot|\mathscr{F}_{\mathcal{D}_N})$.

Now take a $j \in \mathbb{N}$ and consider all $N>j+k$. For each $\mathfrak{B} \in S_{N-j}^\circ(\mathcal{G}_N)$, let $\bc:= \mathfrak{B}
\cap\allrip{N}$; observe that $|\bc|=m_F^{N-j-k}$. Let $\alpha_{x_0,\kappa} :=2(G_{\mathcal{G}_\infty}(x_0,x_0)-\kappa) \log t_F$, where $\kappa \in (0, G_{\mathcal{G}_\infty}(x_0,x_0))$. Finally, for
$\delta\in (0,1) $, define the events
\begin{equation}\label{def:Gamma}
\Gamma_{x_0,\mathfrak{B}} := \left\{\varphi: \left|\{x \in \bc : \mu_x \leq \sqrt{\alpha_{x_0,\kappa} N}\}\right| \geq \delta
|\bc|\right\}, \quad \Gamma_{x_0} = \bigcup_{\mathfrak{B} \in S_{N-j}^\circ(\mathcal{G}_N)} \Gamma_{x_0,\mathfrak{B}}.
\end{equation}

By writing $\Omega_{\mathcal{D}_N}^+$ as the disjoint union of $(\Omega_{\mathcal{D}_N}^+ \cap\Gamma_{x_0})$ and
$(\Omega_{\mathcal{D}_N}^+ \cap\Gamma_{x_0}^c)$, we develop (\ref{eq:P+Uppbound}) further as
\begin{equation}
\mathbb{P}(\Omega_{V_N}^+) \leq \mathbb{E}\left(\prod_{x \in \allrip{N}} \mathbb{P}[\varphi_x\geq
0|\mathscr{F}_{\mathcal{D}_N}]\cdot \mathbbm{1}_{\Omega_{\mathcal{D}_N}^+ \cap\Gamma_{x_0}} \right) +
\mathbb{E}\left(\prod_{x \in \allrip{N}} \mathbb{P}[\varphi_x\geq 0|\mathscr{F}_{\mathcal{D}_N}]\cdot
\mathbbm{1}_{\Omega_{\mathcal{D}_N}^+ \cap\Gamma_{x_0}^c} \right).
\label{ineq:probfieldpositive}
\end{equation}

The claim is that the first term on the RHS of (\ref{ineq:probfieldpositive}) becomes negligible as
$N\to\infty$. Since this result plays an essential role in Section \ref{sec:height}, we record it as a separate
lemma.

\begin{lemma}
Let $\gamma\in (0,1)$. Then for $k$ large enough, there exists a constant $\cc{3.1}(\delta,j,k)$, independent of $N$, such that
\begin{equation}
\mathbb{E}\left(\prod_{x \in \allrip{N}} \mathbb{P}[\varphi_x \geq 0|\mathscr{F}_{\mathcal{D}_N}]\cdot
\mathbbm{1}_{\Omega_{\mathcal{D}_N}^+ \cap\Gamma_{x_0}} \right) \leq \exp\left(-\cc{3.1} m_F^N
t_F^{-N(1-\gamma)}\right).
\end{equation}
\label{lem:firsttermvanishes}
\end{lemma}

\begin{proof}[Proof of Lemma \ref{lem:firsttermvanishes}]
Since $\displaystyle \uparrow \lim_{k\to\infty} G_{\mathring{\mathcal{G}}_k}(x_0,x_0) = G_{\mathcal{G}_\infty}(x_0, x_0)$, 
there exists $\kappa \in (0, G_{\mathcal{G}_\infty}(x_0,x_0))$ such that
\begin{equation}
\frac{1}{2 \log t_F}\frac{\alpha_{x_0,k}}{G_{\mathring{\mathfrak{G}}_k}(x_0,x_0)}=\frac{G_{\mathcal{G}_\infty}(x_0,x_0)-\kappa}{G_{\mathring{\mathcal{G}}_k}(x_0,x_0)} \leq 1-\gamma
\label{ineq2}
\end{equation}
 for all sufficiently large $k$. 
 
Meanwhile, on $\Gamma_{x_0}$,
there exists at least a $\mathfrak{B}^* \in S_{N-j}^\circ(\mathcal{G}_N)$ such that $\Gamma_{x_0,\mathfrak{B}^*}$ holds. Therefore
\begin{eqnarray*}
\prod_{x \in \allrip{N}} \mathbb{P}[\varphi_x\geq 0|\mathscr{F}_{\mathcal{D}_N}] &\leq& \mathbb{P}\left[\varphi_{x_0} -
\mu_{x_0} \geq -\sqrt{\alpha_{x_0,\kappa} N} \bigg| \mathscr{F}_{\mathcal{D}_N}\right]^{\delta |\bstarc|}
~=~ \left[1-\Phi\left(-\frac{\sqrt{\alpha_{x_0,\kappa} N}}{\sqrt{G_{\mathring{\mathcal{G}}_k}(x_0,x_0)}}\right)\right]^{\delta|\bstarc|}\\
&\leq& \left[1- \frac{\sqrt{G_{\mathring{\mathcal{G}}_k}(x_0,x_0)}}{\sqrt{\alpha_{x_0,\kappa} N}}
\exp\left(-\frac{\alpha_{x_0,\kappa} N}{2 G_{\mathring{\mathcal{G}}_k}(x_0,x_0)}\right) \right]^{\delta |\bstarc|}\\
&\leq&\left[1- \frac{1}{\sqrt{2(1-\gamma)N \log t_F}} t_F^{-N(1-\gamma)}\right]^{\delta|\bstarc|}
~\leq~ \exp\left(-c m_F^N t_F^{-N(1-\gamma)}\right) \quad \text{on}~\Omega_{\mathcal{D}_N}^+ \cap \Gamma_{x_0}.
 \end{eqnarray*}
In succession, we used the fact that $\varphi_{x_0}-\mu_{x_0}$ is a centered Gaussian random variable under
$\mathbb{P}(\cdot|\mathscr{F}_{\mathcal{D}_N})$, and applied a standard Gaussian estimate. To finish we used (\ref{ineq2}) and the inequality
$1-x \leq e^{-x}$.
 \end{proof}
 
We turn to estimate the second term on the RHS of (\ref{ineq:probfieldpositive}). The key is to
obtain a lower bound for $\sum_{x \in \bc} \mu_x$ ($\mathfrak{B}\in S_{N-j}^\circ(\mathcal{G}_N)$) on $\Omega_{\mathcal{D}_N}^+ \cap \Gamma_{x_0}^c$. We start by writing
$$ \sum_{x \in \bc} \mu_x = \sum_{\substack{x \in \bc\\ \mu_x > \sqrt{\alpha_{x_0,\kappa} N}}} \mu_x +
\sum_{\substack{x \in \bc \\ \sqrt{\alpha_{x_0,\kappa} N} \geq \mu_x \geq 0}} \mu_x.$$
On $\Omega_{\mathcal{D}_N}^+$, the second summand can be bounded below by $0$. As for the first summand, observe that on $\Gamma_{x_0}^c$, there
are at least $(1-\delta)|\bc|$ many representative interior points $x$ whose $\mu_x$ exceeds
$\sqrt{\alpha_{x_0,\kappa} N}$. Therefore
 $$ \sum_{x \in \bc} \mu_x \geq (1-\delta)|\bc| \sqrt{\alpha_{x_0,\kappa} N} \quad \text{~on~} \Omega_{\mathcal{D}_N}^+ \cap \Gamma_{x_0}^c.$$

Introduce arbitrary nonnegative numbers $f_\mathfrak{B}\geq 0$, $\mathfrak{B}\in S_{N-j}^\circ(\mathcal{G}_N)$. We then have
\begin{eqnarray*}
\mathbb{P}(\Omega_{\mathcal{D}_N}^+ \cap \Gamma_{x_0}^c) &\leq& \mathbb{P}\left( \frac{1}{|\bc|}\sum_{x\in \bc} \mu_x \geq (1-\delta)
\sqrt{\alpha_{x_0,\kappa} N}\right) \\
&=& \mathbb{P}\left(\sum_{\mathfrak{B}\in S_{N-j}^\circ(\mathcal{G}_N)} f_\mathfrak{B} \frac{1}{|\bc|}\sum_{x\in \bc} \mu_x \geq (1-\delta)
\sqrt{\alpha_{x_0,\kappa} N} \sum_{\mathfrak{B}\in S_{N-j}^\circ(\mathcal{G}_N)} f_\mathfrak{B}\right)\\
&\leq& \exp\left[-\frac{\displaystyle (1-\delta)^2 \alpha_{x_0,\kappa} N \left(\sum_{\mathfrak{B}\in S_{N-j}^\circ(\mathcal{G}_N)}
f_\mathfrak{B}\right)^2}{\displaystyle 2 ~{\rm Var} \left(\sum_{\mathfrak{B}\in S_{N-j}^\circ(\mathcal{G}_N)} f_\mathfrak{B} \frac{1}{|\bc|} \sum_{x \in \bc}
\mu_x \right)} \right].
\end{eqnarray*}
In the last inequality, we used the fact that $\sum_{x\in \mathfrak{B}_\mathcal{C}} \mu_x$ is Gaussian (being a sum of independent Gaussians) under $\mathbb{P}$, and applied a standard Gaussian estimate.

The law of total variance, ${\rm Var}(X) = {\rm Var}(\mathbb{E}(X|\mathscr{F})) + \mathbb{E}
({\rm Var}(X|\mathscr{F}))$, implies that  $ {\rm Var}(\mathbb{E}(X|\mathscr{F}))\leq {\rm Var}(X)$. We apply this to
our setting to get
$${\rm Var}\left(\sum_{\mathfrak{B}\in S_{N-j}^\circ(\mathcal{G}_N)} f_\mathfrak{B} \sum_{x \in \bc} \mu_x \right) \leq {\rm Var}\left(\sum_{\mathfrak{B}\in
S_{N-j}^\circ(\mathcal{G}_N)} f_\mathfrak{B} \sum_{x \in \bc} \varphi_x\right).$$

Let the function $\Xi_j : F\to\mathbb{R}_+$ be given by $\displaystyle \Xi_j = \sum_{\mathfrak{B}\in S_{N-j}^\circ(\mathcal{G}_N)} f_\mathfrak{B}
\mathbbm{1}_{\iota_j(\mathfrak{B})}$, where $\iota_j: S^\circ_{N-j} \to \mathcal{Q}^\circ_{N-j}(F_{N-j})$ is a bijection which maps each subgraph to its corresponding subcube. One verifies that
$$ \sum_{\mathfrak{B}\in S_{N-j}^\circ(\mathcal{G}_N)} f_\mathfrak{B} = \frac{1}{|\bc|} \sum_{x \in \allrip{N}} \Xi_j\left(\frac{x}{\ell_F^N}\right)
\qquad {\rm and} \qquad \sum_{\mathfrak{B}\in S_{N-j}^\circ(\mathcal{G}_N)} f_\mathfrak{B} \sum_{x\in \bc}\varphi_x = \sum_{x \in \allrip{N}}
\Xi_j\left(\frac{x}{\ell_F^N}\right) \varphi_x.$$
Hence
\begin{eqnarray*}
{\rm Var} \left(\sum_{\mathfrak{B}\in S_{N-j}^\circ(\mathcal{G}_N)} f_\mathfrak{B} \frac{1}{|\bc|} \sum_{x\in \bc} \varphi_x \right) &=& \frac{1}{|\bc|^2}
{\rm Var} \left(\sum_{x \in \allrip{N}} \Xi_j\left(\frac{x}{\ell_F^N}\right) \varphi_x\right) \\
&=& \frac{1}{|\bc|^2}
\sum_{x,x'\in \allrip{N}}G_{\mathcal{G}_\infty}(x, x') \Xi_j\left(\frac{x}{\ell_F^N}\right) \Xi_j\left(\frac{x'}{\ell_F^N}\right) \\
 &=&  \frac{1}{|\bc|^2}
\sum_{x,x'\in V_N} G_{\mathcal{G}_\infty}(x, x') \left(\Xi_j \mathbbm{1}_{\ell_F^{-N} \mathcal{C}_N} \right)\left(\frac{x}{\ell_F^N}\right) \left(\Xi_j \mathbbm{1}_{\ell_F^{-N} \mathcal{C}_N}\right)\left(\frac{x'}{\ell_F^N}\right).
\end{eqnarray*}

Putting things together,
$$\mathbb{P}(\Omega_{\mathcal{D}_N}^+ \cap\Gamma_{x_0}^c)\leq \exp\left[-\frac{\displaystyle (1-\delta)^2 \alpha_{x_0,\kappa} N
\left(\sum_{x\in \mathcal{C}_N} \Xi_j\left(\frac{x}{\ell_F^N}\right)\right)^2}{\displaystyle 2\sum_{x,x'\in V_N} G_{\mathcal{G}_\infty}(x, x') \left(\Xi_j \mathbbm{1}_{\ell_F^{-N} \mathcal{C}_N} \right)\left(\frac{x}{\ell_F^N}\right) \left(\Xi_j \mathbbm{1}_{\ell_F^{-N}\mathcal{C}_N}\right)\left(\frac{x'}{\ell_F^N}\right)} \right].$$
It follows from this and Lemma \ref{lem:firsttermvanishes} that 
\begin{eqnarray}
\nonumber && \varlimsup_{N\to\infty} \frac{\log \mathbb{P}(\Omega_{V_N}^+)}{\rho_F^{-N} N \log t_F} ~\leq~
\varlimsup_{N\to\infty} \frac{\log \mathbb{P}(\Omega_{\mathcal{D}_N}^+ \cap\Gamma_{x_0}^c)}{\rho_F^{-N} N \log t_F}\\
\nonumber &\leq& \varlimsup_{N\to\infty} \left[-(1-\delta)^2 (G_{\mathcal{G}_\infty}(x_0,x_0)-\kappa) \frac{\displaystyle
\left(\frac{1}{m_F^{N-k}}\sum_{x\in \mathcal{C}_N} \Xi_j\left(\frac{x}{\ell_F^N}\right)\right)^2}{\displaystyle
\frac{\rho_F^{-N}}{m_F^{2(N-k)}} \sum_{x,x'\in V_N}G_{\mathcal{G}_\infty}(x,x')\left(\Xi_j \mathbbm{1}_{\ell_F^{-N} \mathcal{C}_N}\right)\left(\frac{x}{\ell_F^N}\right) \left(\Xi_j \mathbbm{1}_{\ell_F^{-N} \mathcal{C}_N}\right)\left(\frac{x'}{\ell_F^N}\right) } \right]\\
\nonumber &\leq& -\frac{(1-\delta)^2 (G_{\mathcal{G}_\infty}(x_0,x_0) -\kappa)}{\cc{1.2}^2 m_F^{2k}} \frac{\displaystyle \varliminf_{N\to\infty} \left(\frac{1}{m_F^{N-k}}\sum_{x\in \mathcal{C}_N} \Xi_j\left(\frac{x}{\ell_F^N}\right)\right)^2}{\displaystyle \varlimsup_{N\to\infty}\frac{\rho_F^{-N}}{|V_N|^2} \sum_{x,x'\in V_N}G_{\mathcal{G}_\infty}(x,x')\left(\Xi_j \mathbbm{1}_{\ell_F^{-N} \mathcal{C}_N}\right)\left(\frac{x}{\ell_F^N}\right) \left(\Xi_j \mathbbm{1}_{\ell_F^{-N} \mathcal{C}_N}\right)\left(\frac{x'}{\ell_F^N}\right)}\\
\label{eq:capacityargument}&\leq& - \frac{(1-\delta)^2 (G_{\mathcal{G}_\infty}(x_0,x_0)-\kappa)}{\cc{1.2}^2 \cc{2.6}} \left[\frac{\langle \mathbbm{1}_F,\Xi_j
\nu\rangle_F^2}{\mathcal{E}(U(\Xi_j \nu), U(\Xi_j \nu))} \right]
\end{eqnarray}
for all $\Xi_j \in\mathcal{F}$. In obtaining the convergence for the denominator, we applied Lemma \ref{lem:VNConv}(ii) and identified the limit measure as $m_F^{-k} \Xi_j \nu$.

By varying over the coefficients $f_\mathfrak{B}$ in
$\Xi_j$ and taking the limit $j\to\infty$, we can obtain any $\Xi \nu \in \mathcal{M}_{0,{\rm ac}}^{(0)}(F)$. Then we can recover any $\mu \in S_0^{(0)}$, ${\rm supp}(\mu)\subset F$, by an approximating sequence of measures in $\mathcal{M}_{0,{\rm ac}}^{(0)}(F)$ \emph{\`{a} la} Yosida (Proposition \ref{prop:Yosida}). We supremize the bracketed expression on the RHS of (\ref{eq:capacityargument}) over all $\mu \in S_0^{(0)}$ and apply Proposition \ref{prop:capacity}(iv), then take $\delta,\kappa\to 0$ to get
\begin{equation}
\varlimsup_{N\to\infty} \frac{\log \mathbb{P}(\Omega_{V_N}^+)}{\rho_F^{-N} N \log t_F} \leq
-\frac{1}{\cc{1.2}^2 \cc{2.6}}\cdot G_{\mathcal{G}_\infty}(x_0,x_0)\cdot {\rm Cap}_{\mathcal{E}}(F).
\label{ineq:LDUB}
\end{equation}
This essentially proves the upper bound in Theorem \ref{thm:LD}, though \emph{a priori} not the sharpest possible
bound. In principle, one can choose the interior point $x_0^* \in V_k \backslash \partial \mathcal{G}_k$ with the biggest
on-diagonal Green's function value $G_{\mathcal{G}_\infty}(x_0^*, x_0^*)$, and run through the preceding argument
to get (\ref{ineq:LDUB}) with $G_{\mathcal{G}_\infty}(x_0^*,x_0^*)$ in place of
$G_{\mathcal{G}_\infty}(x_0,x_0)$.

\section{Proof of Theorem \ref{thm:samplemeanheight}} \label{sec:height}

The purpose of this section is to prove that for any $\epsilon>0$ and any $\eta>0$,
\begin{eqnarray}
\lim_{N\to\infty} \sup_{\substack{x\in V_N \\ V_{N,\epsilon}(x) \subset V_N}}
\mathbb{P}\left(\bar\varphi_{N,\epsilon}(x) \leq \left(\sqrt{2\underline{G}\log t_F} -\eta\right) \sqrt{N}~ \bigg|
~\Omega^+_{V_N}\right)&=&0. \label{eq:heightLB} \\
\lim_{N\to\infty} \sup_{\substack{x\in V_N \\ V_{N,\epsilon}(x) \subset V_N}}
\mathbb{P}\left(\bar\varphi_{N,\epsilon}(x) \geq \left(\sqrt{2\underline{G}\log t_F} +\eta\right) \sqrt{N}~ \bigg|
~\Omega^+_{V_N}\right)&=&0. \label{eq:heightUB}
\end{eqnarray}

\subsection{Lower bound} \label{subsec:heightlb}

In this subsection, $\displaystyle \mathcal{L}_S:=\frac{1}{|S|}\sum_{x\in S}\delta_{\varphi_x}$ denotes the
empirical measure of the free field $\varphi$ on a measurable subset $S$ of $V_\infty$.

Equation (\ref{eq:heightLB}) is a direct consequence of the following lemma. 

\begin{lemma}
For any $\alpha < 2 \underline{G}\log t_F$ and $\delta>0$,
\begin{equation}
\lim_{N\to\infty} \mathbb{P}\left(\mathcal{L}_{V_N}\left[0,\sqrt{\alpha N}\right] \geq
\delta~\bigg|~\Omega^+_{V_N} \right)=0.
\label{eq:nolowerheight}
\end{equation}
\label{lem:nolowerheight}
\end{lemma}

\begin{proof}
For the sake of clarity, we present the proof in two steps. 

\textbf{Step 1.} Fix a representative interior point $x_0 \in V_k \backslash \partial \mathcal{G}_k$ as in Section
\ref{subsec:LDUB}. Also recall the definition of $\allrip{N}$. Our interim goal is to show that for any $\alpha<
2G_{\mathcal{G}_\infty}(x_0,x_0) \log t_F$ and $\delta>0$,
\begin{equation}
\lim_{N\to\infty} \mathbb{P}\left(\mathcal{L}_{\allrip{N}}\left[0,\sqrt{\alpha N}\right] \geq
\delta~\bigg|~\Omega^+_{V_N} \right) = 0.
\label{eq:condprobvanish1}
\end{equation}

Following the proof of \cite[Lemma 4.4]{BDZ95}, we define, for each $\alpha>0$, the events
$$
\Theta_N(\alpha) = \left\{x\in \allrip{N}: \varphi_x \leq \sqrt{\alpha N}\right\}\quad \text{and}
\quad\bar\Theta_N(\alpha) = \left\{x\in \allrip{N}: \mu_x \leq \sqrt{\alpha N}\right\}.
$$
Then for each $\delta>\delta'>0$ and $\alpha<\alpha'<2 G_{\mathcal{G}_\infty}(x_0,x_0) \log t_F$,
\begin{eqnarray*}
\left\{\mathcal{L}_{\allrip{N}}\left[0,\sqrt{\alpha N}\right] \geq \delta\right\} &=& \left\{ |\Theta_N(\alpha)|
\geq \delta |\allrip{N}|\right\}\\
&=& \left\{ |\Theta_N(\alpha)| \geq \delta|\allrip{N}|, |\bar\Theta_N(\alpha')| \geq \delta' |\allrip{N}|\right\}
\cup \left\{ |\Theta_N(\alpha)| \geq \delta|\allrip{N}|, |\bar\Theta_N(\alpha')| < \delta' |\allrip{N}|\right\}\\
&\subset& \left\{ |\bar\Theta_N(\alpha')| \geq \delta' |\allrip{N}|\right\} \cup \left\{
|\Theta_N(\alpha)\cap\bar\Theta_N(\alpha')^c| \geq (\delta-\delta')|\allrip{N}|\right\} ~=: J_0 \cup J_1.
\end{eqnarray*}
Observe that $J_0 \subset \Gamma_{x_0}$, where $\Gamma_{x_0}$ was introduced in (\ref{def:Gamma}). Therefore, Lemma \ref{lem:firsttermvanishes} implies that for each $\gamma \in (0,1)$, there exists a positive constant $\cc{3.1}$ such
that 
\begin{eqnarray*}
\mathbb{P}\left(J_0 \cap\Omega^+_{V_N}\right)  &\leq&  \mathbb{E}\left(\prod_{x\in \allrip{N}} \mathbb{P}[\varphi_x\geq 0|\mathscr{F}_{\mathcal{D}_N}]\cdot
\mathbbm{1}_{J_0 \cap \Omega_{\mathcal{D}_N}^+ } \right) \\
&\leq &  \mathbb{E}\left(\prod_{x\in \allrip{N}} \mathbb{P}[\varphi_x\geq 0|\mathscr{F}_{\mathcal{D}_N}]\cdot
\mathbbm{1}_{ \Gamma_{x_0} \cap \Omega_{\mathcal{D}_N}^+} \right) ~\leq~  \exp\left(-\cc{3.1} m_F^N t_F^{-N(1-\gamma)}\right).
\end{eqnarray*}
On the other hand, the lower bound of Theorem \ref{thm:LD} implies that for all sufficiently large $N$,
$\mathbb{P}(\Omega^+_{V_N}) \geq \exp\left(-c \rho_F^{-N} N \log t_F \right)$. Therefore
$$ \frac{\mathbb{P}\left(J_0 \cap\Omega^+_{V_N}\right)}{\mathbb{P}(\Omega^+_{V_N})} \leq \exp\left(-c \rho_F^{-N}
\left(t_F^{N\gamma} -c'N\right)\right),$$
and the RHS tends to $0$ as $N\to\infty$ because $\rho_F<1$. Thus it remains to show that
$$
\frac{\mathbb{P}\left(J_1 \cap\Omega^+_{V_N}\right)}{\mathbb{P}(\Omega^+_{V_N})} \longrightarrow 0 \quad
\text{as} \quad N\to\infty.
$$
Note first that $\mu_x - \varphi_x \geq (\sqrt{\alpha'}-\sqrt{\alpha})\sqrt{N}$ whenever $x\in \Theta_N(\alpha)
\cap\bar\Theta_N(\alpha')^c$. So on $J_1$,
$$\frac{1}{|\allrip{N}|} \sum_{x\in \allrip{N}} |\varphi_x-\mu_x| \geq (\delta-\delta')
(\sqrt{\alpha'}-\sqrt{\alpha})\sqrt{N}.$$
Hence
\begin{eqnarray*}
\mathbb{P}\left(J_1 \cap\Omega^+_{V_N}\right)&\leq& \mathbb{P}\left(J_1 \cap\Omega^+_{\mathcal{D}_N}\right)\\
&\leq&\mathbb{E}\left(\mathbb{P}\left(\frac{1}{|\allrip{N}|} \sum_{x\in \allrip{N}} |\varphi_x-\mu_x| \geq
(\delta-\delta') (\sqrt{\alpha'}-\sqrt{\alpha})\sqrt{N}~\bigg|~ \mathscr{F}_{\mathcal{D}_N} \right)\cdot
\mathbbm{1}_{J_1 \cap\Omega^+_{\mathcal{D}_N}}\right) \\ &\leq& \exp\left(- \frac{(\delta-\delta')^2
(\sqrt{\alpha'}-\sqrt{\alpha})^2 N |\allrip{N}|^2}{2|\allrip{N}| G_{\mathring{\mathcal{G}}_k}(x_0,x_0)}\right)\\
&\leq& \exp\left(-C N \rho_F^{-N} t_F^N\right)
\end{eqnarray*}
for some positive constant $C$ which depends on anything but $N$. In the third inequality above, we used the fact that under $\mathbb{P}(\cdot | \mathscr{F}_{\mathcal{D}_N})$, $\{\varphi_x-\mu_x :x\in \allrip{N}\}$
are independent centered Gaussian random variables with variance $G_{\mathring{\mathcal{G}}_k}(x_0,x_0)$. This shows that $\mathbb{P}(J_1
\cap\Omega^+_{V_N})$ decays faster than $\mathbb{P}(\Omega^+_{V_N})$ as $N\to\infty$, and hence proves
(\ref{eq:condprobvanish1}).


\textbf{Step 2.} Observe that the proof in Step 1 continues to hold for any other interior point $x_0 \in
V_k\backslash \partial \mathcal{G}_k$ with the obvious replacements. Thus we can deduce that for any $\alpha< 2
\left(\min_{x_0 \in V_k \backslash \partial \mathcal{G}_k}G_{\mathcal{G}_\infty}(x_0,x_0)\right) \log t_F$ and $\delta>0$,
$$ \lim_{N\to\infty} \mathbb{P}\left( \mathcal{L}_{V_N \backslash \mathcal{D}_N}\left[0,\sqrt{\alpha N}\right]
\geq \delta ~\bigg|~ \Omega^+_{V_N}\right)=0.$$
This falls short of (\ref{eq:nolowerheight}) because $\mathcal{D}_N$ has been excluded from the empirical
measure. To redress this shortcoming, we need to translate the conditioning grid relative to the underlying graph
$V_\infty$, so that points on $\mathcal{D}_N$ lie within the grid, and then carry out the conditioning scheme.
Let us take a moment to describe the translation procedure, as it will be used again in \S\ref{subsec:heightub}. 

\begin{figure}
\centering
\includegraphics{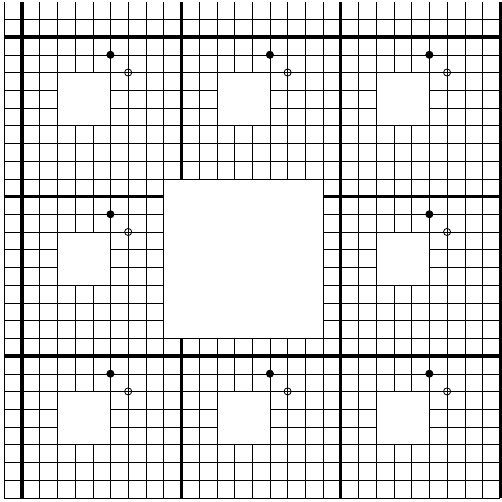}
\caption{The coarse graining and conditioning scheme upon translation. As in Figure \ref{fig:conditioning}, the filled dots indicate the original representative interior points ($\mathcal{C}_N$). Applying a translation by $z-x_0$ for some $z\in V_k$ (one of the hollow dots), one obtains the new representative interior points ($\tilde{\mathcal{C}}_N^z$, hollow dots) and conditioning grid ($\tilde{\mathcal{D}}_N^z$, solid lines).}
\label{fig:conditioningtranslated}
\end{figure}

As before we fix a representative interior point $x_0 \in V_k \backslash \partial \mathcal{G}_k$. For each $z \in [0,\ell_F^k)^d \cap V_k =: V_k^\lefthalfcup$, define
\begin{eqnarray*}
\tallrip{N}^z &=& \{ x\in V_N: x=z+\ell_F^k {\bf p} ~\text{for some}~ {\bf p}\in (\mathbb{N}_0)^d\},\\
\tilde{\mathcal{D}}_N^z &=& \{ x\in V_N: \exists i\in \{1,\cdots,d\} ~\text{such that}~ x_i = p\ell_F^k + (z-x_0)_i ~\text{for some}~ p\in \mathbb{N}_0 \}.
\end{eqnarray*}
In effect, we are translating the set of coarse-graining points and the conditioning grid by a vector $z- x_0$; see Figure \ref{fig:conditioningtranslated}. Since $\tilde{\mathcal{D}}_N^z$ separates points in $\tallrip{N}^z$, we can define finite subgraphs $\mathfrak{g}_x = (V(\mathfrak{g}_x),\sim)$ of $\mathcal{G}_N$ indexed by $x\in \tallrip{N}^z$ such that:
\begin{itemize}
\item The boundary $\partial \mathfrak{g}_x$ of $\mathfrak{g}_x$ is a subset of $\tilde{\mathcal{D}}_N^z$ which separates $x$ from all other elements of $\tallrip{N}^z$.
\item $V(\mathfrak{g}_x) = V(\mathring{\mathfrak{g}_x}) \cup \partial \mathfrak{g}_x$, where $V(\mathring{\mathfrak{g}_x})$ is the set of vertices whose shortest path to $x$ does not intersect $\partial \mathfrak{g}_x$. Here we used the notation $\mathring{\mathfrak{g}_x}$ to indicate the interior of $\mathfrak{g}_x$.
\end{itemize}

The conditioning argument now reads as follows: Under $\mathbb{P}(\cdot| \mathscr{F}_{\tilde{\mathcal{D}}_N^z})$, $\{\varphi_x: x\in \tallrip{N}^z\}$ are independent Gaussian random variables, each having mean
$\mathbb{E}(\varphi_x | \mathscr{F}_{\tilde{\mathcal{D}}_N^z}) =: \tilde\mu_x^z$ and
variance $G_{\mathring{\mathfrak{g}}_x}(x,x)$. Keep in mind that the variances are \emph{not
all} identical because the subgraphs $(\mathfrak{g}_x)_{x\in \tallrip{N}^z}$ no longer retain the symmetries of the original carpet. Nevertheless, we still have the resistance shorting rule
$$ G_{\mathring{\mathfrak{g}}_x}(x,x) = R_{\rm eff}\left(x,V\left(\left(\mathring{\mathfrak{g}}_x\right)^c\right)\right) \leq R_{\rm eff}(x,\{\infty\}) = G_{\mathcal{G}_\infty}(x, x),$$
where $R_{\rm eff}(A,B)$ is the effective resistance between two (finite) subsets $A$, $B$ of $V_\infty$ on the graph $\mathcal{G}_\infty$.

Now define $\tbc^z= V(\mathfrak{B}) \cap \tallrip{N}^z$ for each $\mathfrak{B} \in S_{N-j}^\circ(\mathcal{G}_N)$ and each $z\in V_k^\lefthalfcup$. Note that $V(\mathfrak{B})$ equals the disjoint union $\bigcup_{z\in V_k^\lefthalfcup}\tbc^z$, and that the $|\tbc^z|$ are not the same for all $z$ and $\mathfrak{B}$ due to inclusion/exclusion of $k$th-level boundary points. Nevertheless we still have $|\tbc^z|=\mathcal{O}(m_F^{N-k})$.

Let $\kappa>0$ and $\alpha_\kappa :=2
(\underline{G}-\kappa)\log t_F$. Define, for $\delta\in (0,1)$, the event 
\begin{equation}
\tilde\Gamma_{\mathfrak{B}}^z:=\left\{\varphi:
\left|\{ x\in \tbc^z: \tilde\mu_x^z \leq \sqrt{\alpha_\kappa N}\right| \geq \delta |\tbc^z|\right\},
\label{eq:event}
\end{equation}
 and put
$\tilde\Gamma^z = \bigcup_{\mathfrak{B}\in S_{N-j}^\circ(\mathcal{G}_N)} \tilde\Gamma_{\mathfrak{B}}^z$. We have the following analog of Lemma
\ref{lem:firsttermvanishes}:
\begin{lemma}
Let $\gamma\in (0,1)$. Then for $k$ large enough, there exists a constant $\cc{4.1}(\delta,k)$, independent of $N$, such that
\begin{equation}
\mathbb{E}\left(\prod_{x \in \tallrip{N}^z} \mathbb{P}[\varphi_x \geq 0|\mathscr{F}_{\tilde{\mathcal{D}}_N^z}]\cdot
\mathbbm{1}_{\Omega_{\tilde{\mathcal{D}}_N^z}^+ \cap \tilde\Gamma^z} \right) \leq \exp\left(-\cc{4.1} m_F^N
t_F^{-N(1-\gamma)}\right).
\end{equation}
\label{lem:firsttermvanishes2}
\end{lemma}
The proof is essentially identical to that of Lemma \ref{lem:firsttermvanishes}, except that we cannot peg the
height to be anything higher than $\sqrt{\alpha_\kappa N}$ in the event $\tilde\Gamma_{\mathfrak{B}}^z$, due to the unequal
variances amongst the conditioned variables.

At last we can describe how to adapt the proof in Step 1 to the translated conditioning grid. The events
$\Theta_N(\alpha)$, $\bar\Theta_N(\alpha)$, $J_0$ and $J_1$ are as before, except that one replaces
$\mathcal{C}_N$, $\mathcal{D}_N$, and $\mu_x$ with, respectively, $\tilde{\mathcal{C}}_N^z$,
$\tilde{\mathcal{D}}_N^z$, and $\tilde\mu_x^z$, and puts $\alpha<\alpha' < 2\underline{G}\log t_F$. Then by the
aforementioned conditioning argument and Lemma \ref{lem:firsttermvanishes2}, one shows that
$\lim_{N\to\infty}\mathbb{P}(J_0 | \Omega^+_{V_N}) =0$. Similarly, using conditioning and a standard Gaussian
estimate, one finds $\lim_{N\to\infty}\mathbb{P}(J_1 | \Omega^+_{V_N})=0$. Upon varying over all $z\in V_k^\lefthalfcup$ one proves Lemma \ref{lem:nolowerheight}.
\end{proof}

\subsection{Upper bound} \label{subsec:heightub}

In this subsection we prove the upper bound (\ref{eq:heightUB}). The overall strategy is to show that on $\Omega^+_{V_N}$, the coarse-grained averages of
$\varphi_x$ and of $\mu_x$ differ by $\mathcal{O}(1)$ as $N\to\infty$. Since the $\mu_x$ are independent under the
conditioning, we can use standard Gaussian estimates to bound them below uniformly by a threshold
$\sqrt{2\underline{G} \log t_F N}$. It follows that the local sample mean of the actual field $\varphi_x$ is bounded
below by the same threshold plus an $\mathcal{O}(1)$ error. Finally, we invoke the convergence of discrete Green forms (Lemma \ref{lem:VNConv}(ii)) and a capacity argument (like the one used at the end of the proof in \S\ref{subsec:LDUB}) to establish the asymptotic sharpness of the threshold. Our approach is inspired by \cite{Kurt}.

In what follows, we fix an $y\in F$ and an $\epsilon>0$ such that the $\epsilon$ cubic neighborhood of $y$,
$$\eucball=\left\{y'\in F: \max_{1\leq i \leq d} |y'_i-y_i| \leq \epsilon\right\},$$
is contained in $F$. We then let 
$$V_{N,\epsilon}(y) = \left\{z \in V_N: \max_{1\leq i \leq d} \left|z_i-[ \ell_F^N y_i ]\right| \leq \epsilon
\cdot \ell_F^N \right\},$$
and denote by $\mathcal{G}_{N,\epsilon}(y)= (V_{N,\epsilon}(y),\sim)$ the corresponding graph. In essence, $\mathcal{G}_{N,\epsilon}(y)$ (relative to $\mathcal{G}_N$) can be viewed as the graphical approximation of $\eucball$ (relative to $F$). 

The quantity of interest is the average of $\varphi$ over $V_{N,\epsilon}(y)$, \emph{i.e.,} the local sample mean of the free field,
 $$\displaystyle \bar{\varphi}_{N,\epsilon}(y) = \frac{1}{|V_{N,\epsilon}(y)|} \sum_{z \in
V_{N,\epsilon}(y)} \varphi_z.$$ For each $\eta>0$, denote $$\mneta := \left\{\bar{\varphi}_{N,\epsilon}(y)\geq
(\sqrt{2 \underline{G} \log t_F} +\eta)\sqrt{N} \right\}.$$
Our goal is to prove that for any $\eta>0$,
\begin{equation}
\lim_{N\to\infty}\mathbb{P}\left(\mneta~|~ \Omega^+_{V_N}\right)= 0.
\label{eq:excessheight}
\end{equation}

To begin the proof, we fix a sufficiently large $k \in \mathbb{N}$, take a $j\in \mathbb{N}$, and consider all $N>k+j$. Fix
a representative interior point $x_0 \in V_k \backslash \partial \mathcal{G}_k$ as usual. Let $\kappa > 0$, and denote
$\alpha_\kappa=2(\underline{G}-\kappa)\log t_F$. Two events are introduced as follows. The first event is $\tilde\Gamma^z = \bigcup_{\mathfrak{B} \in S_{N-j}^\circ(\mathcal{G}_N)} \tilde\Gamma_{\mathfrak{B}}^z$, where $\tilde\Gamma_{\mathfrak{B}}^z$ is given in (\ref{eq:event}) and is defined for each $\mathfrak{B}\in S_{N-j}^\circ(\mathcal{G}_N)$, $z\in V_k^\lefthalfcup$ and $\delta\in (0,1)$. The second event, defined for each $s>0$ and $z\in V_k^\lefthalfcup$, is
$$\tilde{D}_s^z := \left\{\varphi: \text{there exists an}~ \mathfrak{B} \in S_{N-j}^\circ(\mathcal{G}_N) \text{~such that~}
\frac{1}{|\tbc^z|}\sum_{x\in \tbc^z} (\varphi_x-\tilde\mu_x^z) < -s\right\}.$$

Observe that $\mneta$ equals the disjoint union $(\mneta \cap J_2) \cup (\mneta \cap J_3) \cup (\mneta \cap
J_4)$, where
$$
J_2 := \left(\bigcup_{z \in V_k^\lefthalfcup}\tilde\Gamma^z \right),\quad
J_3 := \left(\bigcap_{z \in V_k^\lefthalfcup}\left(\tilde\Gamma^z\right)^c \right)\cap\left(\bigcup_{z \in V_k^\lefthalfcup}\tilde
D_s^z\right),\quad
J_4 := \left(\bigcap_{z \in V_k^\lefthalfcup}\left(\tilde\Gamma^z\right)^c \right)\cap\left(\bigcap_{z \in V_k^\lefthalfcup} \left(\tilde
D_s^z\right)^c \right).
$$
So the task boils down to proving that each of $\mathbb{P}(\mneta \cap J_2 \cap\Omega^+_{V_N})$,
$\mathbb{P}(\mneta \cap J_3 \cap\Omega^+_{V_N})$, and $\mathbb{P}(\mneta \cap J_4 \cap\Omega^+_{V_N})$ decays
faster than $\mathbb{P}(\Omega_{V_N}^+)$ as $N\to\infty$.

For $J_2$, we combine Lemma \ref{lem:firsttermvanishes2} with a union bound to find that for any $\gamma\in
(0,1)$,
$$\mathbb{P}(J_2 \cap\Omega^+_{V_N}) \leq |V_k^\lefthalfcup| \exp\left(-Cm_F^N t_F^{-N(1-\gamma)}\right),$$
which decays faster than $\mathbb{P}(\Omega^+_{V_N})$.

For $J_3$, we use the fact that under $\mathbb{P}(\cdot| \mathscr{F}_{\tilde{\mathcal{D}}_N^z})$, $\{\varphi_x -
\tilde\mu_x^z: x \in \tbc^z\}$ are independent (though not identically distributed) Gaussian random variables to
find
$$
\mathrm{Var}\left( \frac{1}{|\tbc^z|} \sum_{x \in \tbc^z} \left( \varphi_x - \tilde\mu_x^z\right) \bigg|
\mathscr{F}_{\tilde{\mathcal{D}}_N^z}\right) = \frac{1}{|\tbc^z|^2} \sum_{x \in \tbc^z} \mathrm{Var} \left(
\varphi_x- \tilde\mu_x^z \bigg|\mathscr{F}_{\tilde{\mathcal{D}}_N^z}\right) \leq \frac{1}{|\tbc^z|^2} \cdot
|\tbc^z|\overline{G} = \frac{\overline{G}}{|\tbc^z|}.
$$
By applying a union bound followed by a Gaussian estimate, we see that there exists $z' \in V_k^\lefthalfcup$ such that
\begin{eqnarray*}
\mathbb{P}(J_3 \cap \Omega^+_{V_N}) &\leq& |V_k^\lefthalfcup| \cdot \mathbb{P}\left( \left(\tilde\Gamma^{z'}\right)^c
\cap\tilde D_s^{z'} \cap\Omega^+_{V_N} \right) \\
&\leq& |V_k^\lefthalfcup| \cdot \mathbb{E}\left( \mathbb{P} \left(\exists \mathfrak{B} \in S_{N-j}^\circ(\mathcal{G}_N): \frac{1}{|\tbc^{z'}|} \sum_{x
\in \tbc^{z'}} \left( \varphi_x - \tilde\mu_x^{z'}\right)<-s ~\bigg|~\mathscr{F}_{\tilde{\mathcal{D}}_N^{z'}}
\right)\cdot \mathbbm{1}_{\Omega^+_{\tilde{\mathcal{D}}_N^{z'}} \cap (\tilde \Gamma^{z'})^c} \right) \\
&\leq& |V_k^\lefthalfcup| \cdot \exp\left( \frac{-C s^2  m_F^N}{2 \overline{G}}\right),
\end{eqnarray*}
where $C$ depends on anything but $N$. This decays faster than $\mathbb{P}(\Omega^+_{V_N})$ as $N\to\infty$.
 
It remains to estimate $\mathbb{P}(\mneta \cap J_4 \cap\Omega^+_{V_N})$. Observe that for every $z \in V_k^\lefthalfcup$,
$\mathfrak{B}\in S_{N-j}^\circ(\mathcal{G}_N)$, and $s>0$,
\begin{eqnarray*}
\frac{1}{|\tbc^z|} \sum_{x \in \tbc^z} (\varphi_x-\tilde\mu_x^z) \geq -s &\text{on}& (\tds^z)^c \cap
\Omega^+_{V_N},\\
\frac{1}{|\tbc^z|} \sum_{x \in \tbc^z} \tilde\mu_x^z \geq (1-\delta)\sqrt{\alpha_\kappa N} 
&\text{on}& (\tilde\Gamma^z)^c \cap \Omega^+_{V_N}.
\end{eqnarray*}
Therefore
$$
\frac{1}{|\tbc^z|} \sum_{x \in \tbc^z} \varphi_x = \frac{1}{|\tbc^z|} \sum_{x \in \tbc^z} \tilde\mu_x^z +
\frac{1}{|\tbc^z|} \sum_{x \in \tbc^z} \left( \varphi_x - \tilde\mu_x^z \right) \geq (1-\delta)\sqrt{\alpha_\kappa
N} - s \quad \text{on}\quad (\tilde \Gamma^z)^c \cap(\tilde D_s^z)^c \cap \Omega^+_{V_N}.
$$
We then take the intersection over all $z \in V_k^\lefthalfcup$ to conclude that for every $\mathfrak{B}\in S_{N-j}^\circ(\mathcal{G}_N)$,
\begin{equation} \label{equ20}
\bar{\varphi}_{\mathfrak{B}}:= \frac{1}{|V(\mathfrak{B})|} \sum_{x \in V(\mathfrak{B})} \varphi_x \geq (1-\delta)\sqrt{\alpha_\kappa N} -s \quad \text{on} \quad
J_4 \cap \Omega^+_{V_N}.
\end{equation}
From now on put $s=\mathcal{O}(1)$.

Motivated by \cite[\S 3]{Kurt}, we define, for each $\theta\in [0,1)$ and $\kappa' >0$, the event
$$
C_{\theta,\kappa'}:= \left\{\varphi: \text{~there exist~} \geq  m_F^{(1-\theta)j} ~\text{many}~ \mathfrak{B}_0 \in S_{N-j}^\circ(\vne) \text{~such that~} \bar{\varphi}_{\mathfrak{B}_0} \geq
\left(\sqrt{\alpha_\kappa}+ m_F^{\theta j} \kappa'\right) \sqrt{N} \right\},$$
where
$$S_{N-j}^\circ(\mathcal{G}_{N,\epsilon}(y)) = \left\{\mathfrak{B} \in S_{N-j}^\circ(\mathcal{G}_N) : \mathfrak{B} \subset \mathcal{G}_{N,\epsilon}(y)\right\}.$$
We claim that for every $\eta>0$, there exist $\theta \in [0,1)$ and $\kappa'>0$, both independent of $j$, such that for all sufficiently large $N$,
$$
\mathcal{M}_{N,\eta} \cap J_4 \cap \Omega_{V_N}^+ \quad \text{implies}\quad  \left\{\bar{\varphi}_{\mathfrak{B}} \geq (1- \delta) \sqrt{\alpha_\kappa N}-\mathcal{O}(1)
,~\forall \mathfrak{B} \in \sve{\vne} \right\} \cap C_{\theta,\kappa'}.
$$
The implication of the first event follows from (\ref{equ20}). To see that $\mathcal{M}_{N,\eta}$ implies $C_{\theta,\kappa'}$, we show by contradiction. Suppose for some $\eta>0$, $\mathcal{M}_{N,\eta}$ holds, but $C_{\theta,\kappa'}$ does not hold for any $\theta\in [0,1)$ and $\kappa'>0$. The latter condition says that $\bar\varphi_{\mathfrak{B}} \leq \sqrt{\alpha_\kappa N}$ for all $\mathfrak{B} \in  S_{N-j}^\circ(\mathcal{G}_{N,\epsilon}(y))$. Hence
$$
\bar\varphi_{N,\epsilon}(y) = \frac{1}{|S_{N-j}^\circ(\mathcal{G}_{N,\epsilon}(y))|} \sum_{\mathfrak{B} \in S_{N-j}^\circ(\mathcal{G}_{N,\epsilon}(y))} \bar\varphi_{\mathfrak{B}} \leq \sqrt{\alpha_\kappa N} = \sqrt{2(\underline{G}-\kappa)\log t_F}\sqrt{N}.
$$
However this contradicts the event $\mathcal{M}_{N,\eta}$, which says that $\bar\varphi_{N,\epsilon}(y) \geq (\sqrt{2\underline{G} \log t_F}+\eta)\sqrt{N}$.

Denote by $S_{N-j}^\theta$ the collection of $\mathfrak{B}_0$ in the event $C_{\theta,\kappa'}$. According to the preceding discussion, for every $\eta > 0$, there exist $\theta\in  [0,1)$ and $\kappa' > 0$, independent of  $j$, such that for all sufficiently large $N$,
\begin{eqnarray*}
&&\mathbb{P}(\mathcal{M}_{N,\eta} \cap J_4 \cap \Omega^+_{V_N}) \\ &\leq& \mathbb{P} \left( \left\{\bar{\varphi}_{\mathfrak{B}} \geq (1- \delta) \sqrt{\alpha_\kappa N}-\mathcal{O}(1)
,~\forall \mathfrak{B} \in \sve{\vne} \right\} \cap C_{\theta,\kappa'} \right) \\
&=&\mathbb{P} \left( \left\{\bar{\varphi}_{\mathfrak{B}} \geq (1- \delta) \sqrt{\alpha_\kappa N}- \mathcal{O}(1), ~\forall \mathfrak{B} \in
\sve{\vne} \right\},~\left\{\bar{\varphi}_{\mathfrak{B}_0} \geq (\sqrt{\alpha_\kappa}+m_F^{\theta j}\kappa')\sqrt{N}, ~\forall \mathfrak{B}_0 \in S_{N-j}^\theta\right\} \right)\\
&\leq& \mathbb{P} \left( \sum_{\mathfrak{B} \in \sve{\vne}} f_{\mathfrak{B}} \bar{\varphi}_{\mathfrak{B}} \geq \left[(1-\delta) \sqrt{\alpha_\kappa
N}-\mathcal{O}(1)\right] \sum_{\mathfrak{B} \in \sve{\vne}} f_{\mathfrak{B}} +  m_F^{(1-\theta)j} \kappa' m_F^{\theta j}\sqrt{N}\right),
\end{eqnarray*}
where in the last line we inserted arbitrary $f_{\mathfrak{B}} \geq 0, \mathfrak{B} \in \sve{\vne} \backslash S_{N-j}^\theta$, and fixed $f_{\mathfrak{B}_0}=1$ for all $\mathfrak{B}_0 \in S_{N-j}^\theta$. Since the $\bar\varphi_{\mathfrak{B}}$ are centered
Gaussian variables, we employ a standard estimate to find
\begin{equation}
\mathbb{P}(\mathcal{M}_{N,\eta} \cap J_4 \cap \Omega^+_{V_N})\leq \exp\left( -\frac{\displaystyle \left( \left[(1-\delta) \sqrt{\alpha_\kappa
N}-\mathcal{O}(1)\right] \sum_{\mathfrak{B} \in \sve{\vne}} f_\mathfrak{B} + \cc{4.2} m_F^j \kappa' \sqrt{N}\right) ^ 2}{\displaystyle 2\mathrm{Var} \left( \sum_{\mathfrak{B}
\in \sve{\vne}} f_\mathfrak{B} \bar{\varphi}_\mathfrak{B}\right)}\right),
\label{eq:probmn}
\end{equation}
for some constant $\cc{4.2}$ independent of $N$ and $j$.

Now let $\Xi_j: F\to \mathbb{R}_+$ be defined by $\displaystyle \Xi_j = \sum_{\mathfrak{B}\in \sve{\vne}} f_\mathfrak{B} \mathbbm{1}_{\iota_{N-j}(\mathfrak{B})}$.
(Note that ${\rm supp}(\Xi_j) \subset \eucball$.) Then
$$
\sum_{x\in V_{N,\epsilon(y)}} \Xi_j\left(\frac{x}{\ell_F^N}\right) = \sum_{x\in V_{N,\epsilon}(y)}\sum_{\mathfrak{B}\in S_{N-j}^\circ(\mathcal{G}_{N,\epsilon}(y))} f_{\mathfrak{B}}\mathbbm{1}_{\iota_{N-j}(\mathfrak{B})}\left(\frac{x}{\ell_F^N}\right) \leq |V_{N-j}| \left(\sum_{\mathfrak{B}\in S_{N-j}^\circ(\mathcal{G}_{N,\epsilon}(y))} f_{\mathfrak{B}}\right),
$$
and
$$
 \sum_{x\in V_{N,\epsilon}(y)}\Xi_j\left(\frac{x}{\ell_F^N}\right) \varphi_x = \sum_{x\in V_{N,\epsilon}(y)}\sum_{\mathfrak{B}\in S_{N-j}^\circ(\mathcal{G}_{N,\epsilon}(y))} f_{\mathfrak{B}} \varphi_x\mathbbm{1}_{\iota_{N-j}(\mathfrak{B})}\left(\frac{x}{\ell_F^N}\right) = \sum_{\mathfrak{B}\in S_{N-j}^\circ(\mathcal{G}_{N,\epsilon}(y))} f_{\mathfrak{B}} \sum_{x\in \mathfrak{B}} \varphi_x.
$$
Consequently,
\begin{eqnarray*}
{\rm Var}\left(\sum_{\mathfrak{B}\in \sve{\vne}} f_\mathfrak{B} \bar\varphi_\mathfrak{B}\right) &\leq& \frac{1}{|V_{N-j}^\lefthalfcup|^2}\sum_{x,x' \in V_{N,\epsilon}(y)}
G_{\mathcal{G}_\infty}(x,x') \Xi_j\left(\frac{x}{\ell_F^N}\right) \Xi_j\left(\frac{x'}{\ell_F^N}\right) \\
&\leq&  \frac{\cc{4.3}}{|V_{N-j}|^2}\sum_{x,x' \in V_N}
G_{\mathcal{G}_\infty}(x,x') \Xi_j\left(\frac{x}{\ell_F^N}\right) \Xi_j\left(\frac{x'}{\ell_F^N}\right)
\end{eqnarray*}
for some constant $\cc{4.3}$ independent of $j$ and $N$. Plugging these into (\ref{eq:probmn}) and then applying Lemma \ref{lem:VNConv}(ii), we find
\begin{eqnarray}
\nonumber&&\varlimsup_{N\to\infty} \frac{\log \mathbb{P}(\mathcal{M}_{N,\eta} \cap J_4 \cap \Omega^+_{V_N})}{\rho_F^{-N} N \log t_F} \\
\nonumber&\leq&  \varlimsup_{N\to\infty}\left[-\frac{\displaystyle
\left(\left[(1-\delta)\sqrt{\underline{G}-\kappa}-\mathcal{O}\left(\frac{1}{\sqrt{N}}\right)\right]\cdot\frac{1}{|V_{N-j}|}\sum_{x\in
V_N} \Xi_j\left(\frac{x}{\ell_F^N}\right)+\frac{\cc{4.2} m_F^j \kappa'}{\sqrt{2\log t_F}}\right)^2}{\displaystyle \cc{4.3} \frac{\rho_F^{-N}}{|V_{N-j}|^2}
\sum_{x,x'\in V_N}G_{\mathcal{G}_\infty}(x,x')\Xi_j\left(\frac{x}{\ell_F^N}\right)\Xi_j\left(\frac{x'}{\ell_F^N}\right) }\right]\\
\nonumber&\leq& - \frac{\displaystyle \varliminf_{N\to\infty}
\left(\left[(1-\delta)\sqrt{\underline{G}-\kappa}-\mathcal{O}\left(\frac{1}{\sqrt{N}}\right)\right]\cdot\frac{1}{|V_N|}\sum_{x\in
V_N} \Xi_j\left(\frac{x}{\ell_F^N}\right)+\frac{\cc{4.4} \kappa'}{\sqrt{2\log t_F}} \right)^2}{\displaystyle \cc{4.3} \varlimsup_{N\to\infty}\frac{\rho_F^{-N}}{|V_N|^2}
\sum_{x,x'\in V_N}G_{\mathcal{G}_\infty}(x,x')\Xi_j\left(\frac{x}{\ell_F^N}\right)\Xi_j\left(\frac{x'}{\ell_F^N}\right) }\\
\label{ineq:conv} &\leq& -\frac{1}{\cc{4.5}} \cdot \frac{\displaystyle
\left((1-\delta)\sqrt{\underline{G}-\kappa}\langle \mathbbm{1}_F,\Xi_j \nu
\rangle_F+\frac{\cc{4.4}\kappa'}{\sqrt{2\log t_F}} \right)^2}{\mathcal{E}(U(\Xi_j \nu), U(\Xi_j \nu))}.
\end{eqnarray}
Note that going from the second to the third line, we multiply both sides of the fraction by $(|V_{N-j}|/|V_N|)^2$, which is of order $m_F^{2j}$ and independent of $N$. So the constants $\cc{4.4}$ is independent of $j$ and $N$. Also $\cc{4.5} = \cc{4.3} \cc{2.6}$ is independent of $j$ and $N$.

Following the end of \S\ref{subsec:LDUB}, one would expect to optimize the coefficients $f_{\mathfrak{B}}$ and take the $j\to\infty$ limit to retrieve the capacity in the RHS of (\ref{ineq:conv}). But we decided prior to (\ref{eq:probmn}) on fixing the coefficients $f_{\mathfrak{B}}=1$ when $\mathfrak{B} \in S^\theta_{N-j}$, while leaving the remaining $f_{\mathfrak{B}}$ variable. Let us write
$$\Xi_j = \left(\sum_{\mathfrak{B} \in S_{N-j}^\circ(\mathcal{G}_{N,\epsilon}(y)) \backslash S_{N-j}^\theta} f_{\mathfrak{B}} \mathbbm{1}_{\iota_{N-j}(\mathfrak{B})}\right) + \sum_{\mathfrak{B}_0\in S_{N-j}^\theta}\mathbbm{1}_{\iota_{N-j}(\mathfrak{B}_0)} =: \Xi_{j,0} + \mathbbm{1}_{\iota_{N-j}(S_{N-j}^\theta)},$$
where the $f_{\mathfrak{B}}\geq 0$ are arbitrary. 

We claim that as $j\to\infty$, uniformly in $N$, the set $\iota_{N-j}(S^\theta_{N-j})$ has vanishing capacity with respect to the Dirichlet form $\mathcal{E}$, so that $\Xi_j$ can be replaced by $\Xi_{j,0}$ in the RHS of (\ref{ineq:conv}). To see this, we apply the Cauchy-Schwarz inequality on the inner product $\mathcal{E}$ to find
\begin{eqnarray*}
&&\mathcal{E}\left(U(\Xi_j\nu), U(\Xi_j\nu)\right)\\
&=& \mathcal{E}\left(U(\Xi_{j,0}\nu), U(\Xi_{j,0}\nu)\right)  + 2\cdot\mathcal{E}\left(U(\Xi_{j,0}\nu), U(\mathbbm{1}_{\iota_{N-j}(S_{N-j}^\theta)}\nu)\right) +\mathcal{E}\left(U(\mathbbm{1}_{\iota_{N-j}(S_{N-j}^\theta)}\nu), U(\mathbbm{1}_{\iota_{N-j}(S_{N-j}^\theta)}\nu)\right)\\
&\leq& \left(\left[\mathcal{E}(U(\Xi_{j,0}\nu), U(\Xi_{j,0}\nu))\right]^{1/2} + \left[\mathcal{E}\left(U(\mathbbm{1}_{\iota_{N-j}(S_{N-j}^\theta)}\nu), U(\mathbbm{1}_{\iota_{N-j}(S_{N-j}^\theta)}\nu)\right)\right]^{1/2}\right)^2\\
&=:& (\sqrt{K_1} + \sqrt{K_2})^2.
\end{eqnarray*}
The key estimate is on $K_2$. Let $Q_j$ denote the closure of any one of the $\iota_{N-j}(\mathfrak{B}_0)$, which is a subset of $\eucball$ contained in a hypercube of side $\ell_F^{-j}$. Also let $G: F_\infty \times F_\infty \to\mathbb{R}_+$ be the integral kernel associated with $U$. By \cite[Corollary 6.13(a)]{BB99}, there exists $\cc{8}$ such that $G(x,x') \leq \cc{8}\|x-x'\|^{d_w-d_h}$ for all $x,x' \in F_\infty$. Since $|S^\theta_{N-j}| \leq m_F^j$, we have the upper bound
\begin{eqnarray*}
K_2 &\leq& m_F^j \int_{Q_j \times Q_j} G(x,x')d\nu(x)d\nu(x')  ~\leq~ C m_F^j \int_{Q_j\times Q_j}\frac{d\nu(x) d\nu(x')}{\|x-x'\|^{d_h-d_w}} \\ &\leq& C m_F^j \int_{Q\times Q}\frac{d\nu(\ell_F^{-j}x)d\nu(\ell_F^{-j}x')}{(\ell_F^{-j}\|x-x'\|)^{d_h-d_w}} 
~\leq~ C m_F^j \ell_F^{-j (d_h+d_w)} \int_{Q\times Q} \frac{d\nu(x)d\nu(x')}{\|x-x'\|^{d_h-d_w}} ~=~\mathcal{O}(t_F^{-j}),
\end{eqnarray*}
where $Q$ is a cubic region of side $\mathcal{O}(1)$, and the last integral is finite by the same argument as in the proof of Lemma \ref{lem:VNConv}(i). This shows the vanishing capacity of $\iota_{N-j}(S^\theta_{N-j})$ uniformly in $N$. Hence as $j\to\infty$,
\begin{equation}
\mathcal{E}(U(\Xi_j\nu), U(\Xi_j\nu)) = \mathcal{E}(U(\Xi_{j,0}\nu), U(\Xi_{j,0}\nu)) + \mathcal{O}(t_F^{-j/2})
\label{eq:approxgreen}
\end{equation}
uniformly in $N$.

We can now bound the RHS of (\ref{ineq:conv}) from above. First use the trivial inequality $\langle \mathbbm{1}_F, \Xi_j\nu\rangle_F \geq \langle \mathbbm{1}_F, \Xi_{j,0}\nu\rangle_F$ and apply it to the numerator. Then plugging (\ref{eq:approxgreen}) into the denominator, we get
\begin{eqnarray}
\varlimsup_{N\to\infty} \frac{\log \mathbb{P}(\mathcal{M}_{N,\eta} \cap J_4 \cap \Omega^+_{V_N})}{\rho_F^{-N} N \log t_F}  
&\leq&  -\frac{1}{\cc{4.5}} \cdot \frac{\displaystyle
\left((1-\delta)\sqrt{\underline{G}-\kappa}\langle \mathbbm{1}_F,\Xi_{j,0} \nu
\rangle_F+\frac{\cc{4.4} \kappa'}{\sqrt{2\log t_F}} \right)^2}{\mathcal{E}(U(\Xi_{j,0} \nu), U(\Xi_{j,0} \nu))}.
\label{eq:finalupperbound}
\end{eqnarray}
Then we can take $j$ arbitrarily large, so that $\Xi_{j,0}$ is a nonnegative function supported on $\eucball$ minus a set of arbitrarily small capacity.  By the Yosida approximation (Proposition \ref{prop:Yosida}), we may then replace $\Xi_{j,0} \nu$ by any $\mu \in S_0^{(0)}$ with the same maximal support set.

Now comes the simple but crucial rescaling argument:
$$
\mathcal{E}(U\mu, U\mu) =\gamma \cdot (\gamma\mathcal{E})(U^\gamma \mu, U^\gamma \mu)\quad \text{for each}~\gamma>0~\text{and for all}~\mu\in S_0^{(0)},
$$
where $U^\gamma = \gamma^{-1}U$ is the $0$-order potential operator associated with $\gamma\mathcal{E}$. Thus the RHS of (\ref{eq:finalupperbound}) can be rewritten as 
$$
-\frac{1}{\gamma\cc{4.3}} \cdot \frac{\displaystyle
\left((1-\delta)\sqrt{\underline{G}-\kappa}\langle \mathbbm{1}_F, \mu
\rangle_F+ \frac{\cc{4.4} \kappa'}{\sqrt{2\log t_F}}\right)^2}{(\gamma\mathcal{E})(U^\gamma \mu, U^\gamma \mu)}.
$$
Fix a compact set $\mathcal{K}$ within $\eucball$ minus the aforementioned set of capacity zero, and choose $\mu$ to be $\mu_\mathcal{K}$, the $0$-order equilibrium measure of $\mathcal{K}$ with respect to $\gamma \mathcal{E}$.  According to Proposition \ref{prop:capacity}, 
$\langle \mathbbm{1}_F,\mu_{\mathcal{K}}\rangle_F = (\gamma\mathcal{E})(U^\gamma \mu_\mathcal{K}, U^\gamma \mu_\mathcal{K})={\rm
Cap}_{\gamma\mathcal{E}}(\mathcal{K})$. Upon taking $\delta,\kappa\to 0$, and combining with the lower bound of
Theorem \ref{thm:LD} (see also Remark \ref{rem:LDLB}), we arrive at
\begin{eqnarray}
\nonumber
&& \varlimsup_{N\to\infty} \frac{\log\mathbb{P}(\mneta \cap J_4 | \Omega^+_{V_N})}{\rho_F^{-N}N\log t_F}\\ \nonumber
&\leq& \varlimsup_{N\to\infty} \frac{\log \mathbb{P}(\mathcal{M}_{N,\eta} \cap J_4 \cap \Omega^+_{V_N})}{\rho_F^{-N} N \log t_F} - \varliminf_{N\to\infty}
\frac{\log \mathbb{P}(\Omega^+_{V_N})}{\rho_F^{-N} N \log t_F}\\
\label{ineq:condprobheight} &\leq& -\frac{1}{\gamma\cc{4.3}}\left[\underline{G}{\rm Cap}_{\gamma\mathcal{E}}(\mathcal{K}) + \frac{\sqrt{2 \underline{G}}
(\cc{4.4}\kappa') }{\sqrt{\log t_F}} + \frac{(\cc{4.4}\kappa')^2}{2\log t_F {\rm
Cap}_{\gamma\mathcal{E}}(\mathcal{K})}\right] + \gamma^{-1} \cc{1.3} \overline{G} {\rm Cap}_{\gamma\mathcal{E}}(F).
\end{eqnarray}
Solving a quadratic inequality shows that the RHS of (\ref{ineq:condprobheight}) is negative if
$$\kappa' > \cc{4.4}^{-1}\cdot{\rm Cap}_{\gamma\mathcal{E}}(\mathcal{K})\cdot\sqrt{2 \log t_F} \cdot \left(\sqrt{\cc{1.3} \cc{4.3} \overline{G} \cdot\frac{{\rm Cap}_{\gamma\mathcal{E}}(F)}{{\rm Cap}_{\gamma \mathcal{E}}(\mathcal{K})}}-\sqrt{\underline{G}} \right).$$
Observe that ${\rm Cap}_{\gamma \mathcal{E}}(\mathcal{K})$ depends linearly on $\gamma$, while the rest of the expression on the RHS is manifestly independent of $\gamma$. So by tuning $\gamma$, we can make $\kappa' > \Delta$ for any
$\Delta>0$, and thus
$$\varlimsup_{N\to\infty} \frac{\log \mathbb{P}(\mneta \cap J_4| \Omega^+_{V_N})}{\rho_F^{-N} N \log t_F}< 0
$$
for any $\eta>0$. This proves (\ref{eq:excessheight}) for any $y\in F$ and $\epsilon>0$, whence
(\ref{eq:heightUB}).

\begin{appendices}

\section{Appendix}

\subsection{Smooth measures and capacity}

In this subsection we introduce the concepts of smooth measures and ($0$-order) capacity with respect to a (transient) regular Dirichlet form. Much of this can be found in \cite[Chapter 2]{FOT} and \cite[Chapter 2]{ChenFukushima}.

Suppose $(\mathcal{E},\mathcal{F})$ is a regular Dirichlet form on $L^2(X,m)$. Let $\mathfrak{O}$ denote the
family of all open subsets of $X$, and for each $A\in\mathfrak{O}$, define $\mathcal{L}_A = \{u\in \mathcal{F}:
u\geq 1 ~m\text{-a.e. on}~A\}$. The $1$-capacity of the set $A \in \mathfrak{O}$ with respect to $\mathcal{E}$ is given by
\begin{equation}
{\rm Cap}_{\mathcal{E},1}(A) =\left\{ \begin{array}{ll} \inf_{f\in \mathcal{L}_A} \left[\mathcal{E}(f,f)+\|f\|_{L^2}^2\right],& \mathcal{L}_A
\neq \emptyset \\ \infty,& \mathcal{L}_A=\emptyset \end{array} \right. .
\label{eq:defcapacity}
\end{equation}
If $A\subset X$ is an arbitrary subset, then put ${\rm Cap}_{\mathcal{E},1}(A) = \inf_{B\in\mathfrak{O},A\subset B} {\rm
Cap}_{\mathcal{E},1}(B)$. A statement is said to hold \emph{quasi-everywhere} (\emph{q.e.}) on $A$ if and only if there exists a
set $U \subset A$ with ${\rm Cap}_{\mathcal{E},1}(U)=0$ such that the statement holds everywhere
on $A\backslash U$. A function $f:X\to\mathbb{R}$ is said to be \emph{quasi-continuous} if for every $\epsilon>0$, there exists an open set $\Omega$ with ${\rm Cap}_{\mathcal{E},1}(\Omega)<\epsilon$ such that $f$ is continuous on $X\backslash \Omega$. We say that $v$ is a \emph{quasi-continuous modification} of $f$ if $v$ is quasi-continuous and $v=f~m$-a.e, and denote $v$ by $\widetilde f$.

A positive Radon measure $\mu$ on $X$ is
called a \emph{measure of finite energy integral} (with respect to $\mathcal{E}$) if there exists a constant
$C_\mu>0$ such that for all $f \in \mathcal{F} \cap C_c(X)$,
\begin{equation}
\int_X |f|d\mu \leq C_\mu \left[\mathcal{E}(f,f)+\|f\|_{L^2}^2\right]^{1/2}.
\label{eq:FEI}
\end{equation}
We denote by $S_0$ the family of all measures of finite energy integral.

If furthermore $(\mathcal{E},\mathcal{F})$ is transient, that is, if there exists a bounded $m$-integrable function $g$, strictly positive $m$-a.e., such that
$$
\int_X \, |u| g\,dm \leq \sqrt{\mathcal{E}(u,u)}, \quad \forall u \in \mathcal{F},
 $$
  then one may complete $\mathcal{F}$ in the $\mathcal{E}$-norm, and  $(\mathcal{F}_e := \overline{\mathcal{F}}^{\mathcal{E}}, \mathcal{E})$ is a Hilbert space called the \emph{extended Dirichlet space}. Consequently, we have the following $0$-order counterparts of the above notions: the \emph{$0$-capacity} of a set $A\in \mathfrak{O}$, denoted by ${\rm Cap}_{\mathcal{E}}(A)$, is given by (\ref{eq:defcapacity}) with $\mathcal{F}$ and $\mathcal{E}(f,f)+\|f\|_{L^2}^2$ replaced respectively by $\mathcal{F}_e$ and $\mathcal{E}(f,f)$. The $0$-capacity of an arbitrary set $A$ then follows similarly. Likewise, a positive Radon
measure $\mu$ on $X$ is called a \emph{measure of finite $0$-order energy integral} if (\ref{eq:FEI}) holds with the same replacements. Denote by $S_0^{(0)}$ the family of all measures of finite $0$-order energy integral.

There is an important connection between $S_0^{(0)}$ and $\mathcal{F}_e$, which is based on the Riesz representation theorem. For every
$\mu \in S_0^{(0)}$, there exists a unique $U\mu \in \mathcal{F}_e$ such that $\mathcal{E}(f,U\mu)= \langle
\widetilde f,\mu\rangle_X$ for all $f \in \mathcal{F}_e$. We shall refer to $U: S_0^{(0)} \to \mathcal{F}_e$ as the $0$-order \emph{potential operator} associated with $\mathcal{E}$. Any $h \in
\mathcal{F}_e$ which can be written in the form $h=U\mu$ for some $\mu \in S_0^{(0)}$ is called a $0$-order
\emph{potential} relative to $\mathcal{E}$.

Let us remark that $S_0^{(0)} \subset S_0 \subset S$, where $S$ is the family of \emph{smooth measures} consisting of all positive Borel measures $\mu$ on $X$ such that:
\begin{itemize}
\item $\mu$ charges no set of zero 1-capacity.
\item There exists an increasing sequence $(F_n)_n$ of closed sets such that $\mu(F_n) <\infty$ for all $n$, and that $\lim_{n\to\infty}{\rm Cap}_{\mathcal{E},1}(K\backslash F_n)=0$ for any compact set $K$.
\end{itemize}
In general, elements of $S_0$ need not be absolutely continuous with respect to $m$, but each of them can be approximated by a sequence of absolutely continuous measures, \emph{cf.} \cite[Lemma 2.2.2]{FOT}. Here we give the $0$-order version of this statement.

\begin{proposition}
Let $(\mathcal{E},\mathcal{F})$ be a transient regular Dirichlet form on $L^2(X,m)$, and let $G_\beta$ and $U$ denote respectively the $\beta$-resolvent and the $0$-order potential operator associated with $\mathcal{E}$. Given each $\mu \in S_0^{(0)}$, let $h_\beta :=\beta(U\mu-\beta G_\beta (U\mu))$ for each $\beta \in \mathbb{N}$. Then as $\beta \to\infty$, $h_\beta \cdot m$ converges vaguely to $\mu$.
\label{prop:Yosida}
\end{proposition}

\begin{proof}
This is the Yosida approximation (\emph{cf.} \cite[(1.3.18)]{FOT}): for all $f\in \mathcal{F}$,
$$
(h_\beta,f)_{L^2(m)} = (\beta(U\mu-\beta G_\beta(U\mu)),f)_{L^2(m)} \underset{\beta\to\infty}{\longrightarrow} \mathcal{E}(U\mu,f).
$$
Therefore $\lim_{\beta\to\infty}\langle f, h_\beta\cdot m\rangle_X = \langle f, \mu\rangle_X$ for all $f\in \mathcal{F} \cap C_c(X)$.
\end{proof}

Last but not least, let us record several equivalent characterizations of the $0$-capacity.

\begin{proposition}
Let $(\mathcal{E},\mathcal{F})$ be a transient regular Dirichlet form on $L^2(X,m)$. Fix an arbitrary set $B \subset X$ and suppose $\mathcal{L}_B \neq \emptyset$.
\begin{enumerate}[label={(\roman*)}]
\item There exists a unique element $e_B$ in $\mathcal{L}_B$ minimizing $\mathcal{E}(\cdot,\cdot)$. In
particular, ${\rm Cap}_{\mathcal{E}}(B) = \mathcal{E}(e_B,e_B)$.
\item $e_B$ is the unique element of $\mathcal{F}_e$ satisfying $\widetilde e_B=1$ q.e. on $B$ and
$\mathcal{E}(e_B,f)\geq 0$ for any $f\in \mathcal{F}_e$ with $\widetilde f\geq 0$ q.e. on $B$.
\item There exists a unique measure $\mu_B \in S_0^{(0)}$ supported in $B$ such that $e_B = U\mu_B$. In
particular, 
$${\rm Cap}_{\mathcal{E}}(B) = \mathcal{E}(U\mu_B, U\mu_B) =\langle
\widetilde{U\mu_B},\mu_B\rangle_X.$$
\item If $B$ is a compact set, then
\begin{eqnarray*}
{\rm Cap}_\mathcal{E}(B) ~=~ \langle \mathbbm{1}_B, \mu_B\rangle_X
&=& \sup\left\{\mathcal{E}(U\mu, U\mu): \mu \in S_0^{(0)},~{\rm supp}(\mu)\subset B,~\widetilde{U\mu} \leq 1~\text{q.e.} \right\}\\
 &=& \sup\left\{\frac{\langle \mathbbm{1}_B,\mu
\rangle_X^2}{\mathcal{E}(U\mu, U\mu)}: \mu\in S_0^{(0)},~{\rm supp}(\mu)\subset B \right\}.
\end{eqnarray*}
\end{enumerate}
\label{prop:capacity}
\end{proposition}

The function $e_B$ and the measure $\mu_B$ are known as, respectively, the $0$-order \emph{equilibrium potential} and
\emph{equilibrium measure} of the set $B$ (with respect to $\mathcal{E}$).

\begin{proof}
The first two items are the 0-order version of \cite[Theorem 2.1.5]{FOT}, as explained on \cite[p.
74]{FOT}. Item (iii) is proved in conjunction with \cite[Lemma 2.2.10]{FOT}. The first two equalities in item (iv) follow
directly from (ii) and (iii), while the third equality can be obtained by the following argument. Since by (ii) and (iii), $e_B = U\mu_B$ and $\widetilde{e}_B=1$ q.e. on $B$, 
\begin{equation}
\langle \mathbbm{1}_B, \mu\rangle_X = \langle \widetilde{e_B}, \mu\rangle_X = \mathcal{E}(U\mu_B, U\mu). 
\end{equation}
Therefore by the Cauchy-Schwarz inequality
$$
\frac{\langle \mathbbm{1}_B,\mu \rangle_X^2}{\mathcal{E}(U\mu, U\mu)}= \frac{\mathcal{E}(U\mu_B, U\mu)^2}{\mathcal{E}(U\mu, U\mu)} \leq \mathcal{E}(U\mu_B, U\mu_B) = {\rm Cap}_\mathcal{E}(B),
$$
with equality holding if and only if $U\mu$ is a constant multiple of $U\mu_B$.
\end{proof}

\subsection{$\Gamma$-convergence and Mosco convergence}

In this subsection we collect some elementary notions of $\Gamma$-convergence and Mosco convergence; see
\cite{DalMaso,Mosco94} for more details. 

Let $\left((\mathcal{E}_N,\mathcal{F}_N)\right)_N$ be a sequence of
closed symmetric quadratic forms on $\mathcal{H}=L^2(X,m)$, where for each $N$, $\mathcal{F}_N = \{f\in \mathcal{H}: \mathcal{E}_N(f,f) <\infty\}$ denotes the natural domain of $\mathcal{E}_N$. Also we shall fix a sequence $(P_N)_N$ of orthogonal projections $P_N: \mathcal{H} \to
\mathcal{F}_N$. Finally, for each $f\in \mathcal{H}$, let $\mathcal{N}(f)$ be the collection of all open neighborhoods of $f$ with respect to the usual topology on $\mathcal{H}$. 

Define
\begin{eqnarray}
(\Gamma\text{-}\varliminf_{N\to\infty} \mathcal{E}_N)(f,f) &=& \sup_{U \in \mathcal{N}(f)}\varliminf_{N\to\infty}
\inf_{h\in U} \mathcal{E}_N(P_N h,P_N h),\label{eq:gammaliminf}\\
(\Gamma\text{-}\varlimsup_{N\to\infty} \mathcal{E}_N)(f,f) &=& \sup_{U \in \mathcal{N}(f)}\varlimsup_{N\to\infty}
\inf_{h\in U} \mathcal{E}_N(P_N h,P_N h).\label{eq:gammalimsup}
\end{eqnarray}
If the liminf
(\ref{eq:gammaliminf}) coincides with the limsup (\ref{eq:gammalimsup}) for all $f\in \mathcal{H}$, then we say
that the sequence $(\mathcal{E}_N)_N$ \emph{\textbf{$\Gamma$-converges}}, and denote the limit by $\displaystyle
\Gamma\text{-}\lim_{N\to\infty}\mathcal{E}_N$.

One may also consider $\Gamma$-convergence with respect to the \emph{weak} topology on $\mathcal{H}$. In particular put
\begin{eqnarray}
(\Gamma_w\text{-}\varliminf_{N\to\infty} \mathcal{E}_N)(f,f) &=& \sup_{U \in \mathcal{N}_w(f)}\varliminf_{N\to\infty} \inf_{h\in U} \mathcal{E}_N(P_N h,P_N
h),\label{eq:wgammaliminf}
\end{eqnarray}
where $\mathcal{N}_w(f)$ denotes the collection of all open neighborhoods of $f$ with respect to the weak topology on $\mathcal{H}$. If the weak liminf \ref{eq:wgammaliminf}) coincides with the limsup (\ref{eq:gammalimsup}), then we say that the sequence
$(\mathcal{E}_N)_N$ \emph{\textbf{converges in the sense of Mosco}}, and denote the limit by $\displaystyle
M\text{-}\lim_{N\to\infty} \mathcal{E}_N$.

The above notions of convergence can be alternatively characterized in terms of sequences in $\mathcal{H}$. It can be
shown (\emph{cf.} \cite[Ch. 8]{DalMaso}) that $\displaystyle\mathcal{E} = \Gamma\text{-}\lim_{N\to\infty}\mathcal{E}_N$ if and only if the following
two conditions hold: 
\begin{enumerate}[label={($\Gamma$\arabic*)}]
\item For every sequence $(u_N)_N$ in $\mathcal{H}$ converging to $u$, $\displaystyle \varliminf_{N\to\infty} \mathcal{E}_N(P_N u_N, P_N
u_N) \geq \mathcal{E}(Pu,Pu)$.
\item For every $u\in \mathcal{H}$, there is a sequence $(u_N)_N$ in $\mathcal{H}$ converging to $u$ such that
$\displaystyle \varlimsup_{N\to\infty} \mathcal{E}_N(P_N u_N,P_N u_N) \leq \mathcal{E}(Pu,Pu)$.
\end{enumerate}

Replace ($\Gamma$1) by the stronger condition
\begin{enumerate}[label={(w$\Gamma$\arabic*)}]
\item For every sequence $(u_N)_N$ in $\mathcal{H}$ converging \emph{weakly} to $u$, $\displaystyle \varliminf_{N\to\infty}
\mathcal{E}_N(P_N u_N, P_N u_N) \geq \mathcal{E}(Pu,Pu)$.
\end{enumerate}
Then $\displaystyle \mathcal{E} = M\text{-}\lim_{N\to\infty}\mathcal{E}_N$ if and only if (w$\Gamma$1) and
($\Gamma$2) hold. Here $P$ is the orthogonal projection onto $\mathcal{D}(\mathcal{E})$.

The relevance of Mosco convergence to our setting is due to the following fact.

\begin{proposition}
For each $\beta>0$, let $G^N_\beta$ and $G_\beta$ respectively denote the $\beta$-resolvent associated with the closed forms $\mathcal{E}_N$ and $\mathcal{E}$. Then the following are equivalent:
\begin{enumerate}[label={(\arabic*)}]
\item (Mosco convergence) $\displaystyle \mathcal{E} = M\text{-}\lim_{N\to\infty} \mathcal{E}_N$.
\item (Strong resolvent convergence in $L^2$) For every $\beta>0$ and every $f\in \mathcal{H}$, $\displaystyle \lim_{N\to\infty} G^N_\beta P_N f=G_\beta P f$ in $L^2$.
\end{enumerate}
\label{prop:MC}
\end{proposition}

\begin{proof}
See, for example, \cite[Theorem 13.6]{DalMaso} or \cite[Theorem 2.4.1]{Mosco94}. 
\end{proof}
 
Let us now specialize to the case where $(\mathcal{E}_N)_N$ is a sequence of regular Dirichlet forms.

\begin{lemma}
Let $((\mathcal{E}_N,\mathcal{F}_N))_N$ be a Mosco convergent sequence of regular Dirichlet forms on $\mathcal{H}$, and assume that the limit form can be extended to a regular Dirichlet form $(\mathcal{E},\mathcal{F})$ on $\mathcal{H}$.
Then for all $\beta>0$ and all $h\in \mathcal{F}$ with $h \geq 0$ $m$-a.e.,
\begin{equation}
\lim_{N\to\infty} \mathcal{E}_N\left(G^N_\beta P_N h, G^N_\beta P_N h\right) = \mathcal{E}(U_\beta
(h\cdot m), U_\beta (h \cdot m)).
\label{eq:B21}
\end{equation}
Furthermore, if all the $((\mathcal{E}_N,\mathcal{F}_N))_N$ and
$(\mathcal{E},\mathcal{F})$ are transient regular Dirichlet forms, then for all $h\in \mathcal{F}_e \cap L^1(X,m)$ with $h\geq 0$ $m$-a.e.,
\begin{equation}
\lim_{N\to\infty} \mathcal{E}_N\left(G^N P_N h
, G^N P_N h\right) = \mathcal{E}(U(h\cdot m), U(h\cdot m)).
\label{eq:B22}
\end{equation}
\label{lem:resolventconv}
\end{lemma}

\begin{proof}
We have the identities
\begin{eqnarray}
\label{eq:greenid1}
\mathcal{E}_N(G^N_\beta P_N f, P_N h)&=& (P_N f,P_N h)_{L^2}- \beta (G^N_\beta P_N f, P_N
h)_{L^2},
\\
\label{eq:greenid2}\mathcal{E}(G_\beta f,  h)&=& (f, h)_{L^2}- \beta (G_\beta  f, 
h)_{L^2}
\end{eqnarray}
for all $f,h\in \mathcal{F}$.
By Proposition \ref{prop:MC}, $G_\beta^N P_N h \to G_\beta Ph$ in $L^2$. Also $P_N h \to Ph = h$ in $L^2$. So taking the limit $N\to\infty$ on (\ref{eq:greenid1}) we get
\begin{eqnarray*}
\lim_{N\to\infty}\mathcal{E}_N(G_\beta^N P_N h, G_\beta^N P_N h) &=& \lim_{N\to\infty} \left[ (P_N h, G_\beta^N P_N h)_{L^2} -\beta(G_\beta^N P_N h, G_\beta^N P_N h)_{L^2}\right]\\
&=& (h, G_\beta h)_{L^2} - \beta(G_\beta h, G_\beta h)_{L^2}\\
&=& \mathcal{E}(G_\beta h, G_\beta h) ~=~ \mathcal{E}(U_\beta(h\cdot m), U_\beta(h\cdot m))
\end{eqnarray*}
using (\ref{eq:greenid2}) in the end. In the transient case, we have the $0$-order version of (\ref{eq:B21}), which is (\ref{eq:B22}).
\end{proof}

\end{appendices}

\bibliographystyle{amsplain}

\vspace{20pt}

\begin{small}
\begin{flushleft}
\textbf{Joe P. Chen}\\
\textsc{Departments of Mathematics, University of Connecticut, Storrs, CT 06269, USA.}\\
\textsc{E-mail:} \href{mailto:joe.p.chen@uconn.edu}{joe.p.chen@uconn.edu}

\vspace{5pt}

\textbf{Baris Evren Ugurcan}\\
\textsc{Department of Mathematics, The University of Western Ontario, London, Ontario, Canada, N6A 5B7.}\\
\textsc{E-mail:} \href{mailto:beu4@cornell.edu}{beu4@cornell.edu}
\end{flushleft}
\end{small}

\end{document}